\documentclass[11pt]{article}

\usepackage{fullpage}
\usepackage{amsthm, amsmath, amsfonts, amssymb}
\usepackage{bbm}
\usepackage{mathrsfs}

\usepackage[colorlinks,citecolor=blue,urlcolor=blue]{hyperref}
\usepackage{mathtools}
\usepackage{authblk}

\usepackage{todonotes}
\setuptodonotes{fancyline, color=red!30, size=\tiny}
\title{Talagrand's convolution conjecture up to $\log \log$ \\ via perturbed reverse heat}
\author{Yuansi Chen}
\affil{ETH Z\"urich}
\date{}

\newcommand{\supp}{\text{supp}}
\newcommand{\excess}{\alpha}

\newcommand{\xii}{\xi^{[i]}}
\newcommand{\flip}{\varsigma}
\newcommand{\cS}{\mathcal{S}}

\newcommand{\fA}{\mathfrak{A}}

\newcommand{\cM}{\mathcal{M}}

\newcommand{\btau}{\kappa}

\newcommand{\bdelta}{\bar{\delta}}
\newcommand{\stopT}{{\boldsymbol{\mathfrak{T}}}}

\newcommand{\sF}{\mathscr{F}}

\newcommand{\fU}{U}

\newcommand{\fL}{L^U}

\newcommand{\uV}{V}
\newcommand{\uW}{W}
\newcommand{\LV}{L^V}

\newcommand{\cG}{\mathcal{G}}

\newcommand{\To}{{T_o}}

\newcommand{\psiW}{\psi}
\newcommand{\PsiW}{\Psi}

\newcommand{\psiV}{\psi}
\newcommand{\PsiV}{\Psi}

\newcommand{\rV}{\mathscr{V}}
\newcommand{\rW}{\mathscr{W}}

\newcommand{\jL}{\bar{L}}

\newcommand{\Ni}{N^{[i]}}
\newcommand{\Nj}{N^{[j]}}
\newcommand{\wNi}{\widetilde{N}^{[i]}}
\newcommand{\wNj}{\widetilde{N}^{[j]}}

\newcommand{\NUi}{N^{U,[i]}}

\newcommand{\sFU}{\mathscr F^U}

\newcommand{\cE}{\mathscr{E}}
\newcommand{\dirichlet}{\mathcal{E}}

\newcommand{\tagone}{{[1]}}
\newcommand{\tagi}{{[i]}}

\newtheoremstyle{named}{}{}{\itshape}{}{\bfseries}{.}{.5em}{\thmnote{#3's }#1}
\theoremstyle{named}

\theoremstyle{plain}
\newtheorem{theorem}{Theorem}
\newtheorem{proposition}{Proposition}
\newtheorem{lemma}{Lemma}

\newtheorem{conjecture}{Conjecture}
\newtheorem{remark}{Remark}

\newcommand{\defn}{:=}

\DeclareMathOperator{\sign}{sign}

\newcommand{\real}{\ensuremath{\mathbb{R}}}
\newcommand{\naturalnum}{\ensuremath{\mathbb{N}}}

\newcommand{\Ind}{\ensuremath{\mathbb{I}}}
\newcommand{\Exs}{\ensuremath{{\mathbb{E}}}}
\newcommand{\Prob}{\ensuremath{{\mathbb{P}}}}
\newcommand{\Law}{\mathcal{L}}

\newcommand{\Ent}{\text{Ent}}
\DeclareMathOperator{\Var}{Var}

\DeclarePairedDelimiterX{\infdivx}[2]{(}{)}{%
  #1\;\delimsize\|\;#2%
}

\newcommand{\tvdist}[2]{\ensuremath{ d_{\text{\tiny{TV}}}\parenth{#1, #2} }}

\newcommand{\brackets}[1]{\left[ #1 \right]}
\newcommand{\parenth}[1]{\left( #1 \right)}

\newcommand{\braces}[1]{\left\{ #1 \right \}}
\newcommand{\abss}[1]{\left| #1 \right |}
\newcommand{\angles}[1]{\left\langle #1 \right \rangle}
\newcommand{\ceils}[1]{\left\lceil #1 \right \rceil}
\newcommand{\floors}[1]{\left\lfloor #1 \right \rfloor}
\newcommand{\tp}{^\top}

\newcommand{\vecnorm}[2]{\left\| #1\right\|_{#2}}

\begin{document}

\maketitle

\begin{abstract}
  We prove that under the heat semigroup $(P_\tau)$ on the Boolean hypercube, any nonnegative function exhibits a uniform tail bound that is better than Markov's inequality. Specifically, for any $\tau > 0$, $n \geq 1$, $\eta > e^3$, and $f: \{-1,1\}^n \to \mathbb{R}_+$ with $\int f d\mu > 0$, we have
  \begin{align*}
    \mathbb{P}_{X \sim \mu}\left( P_\tau f(X) > \eta \int f d\mu \right) \leq c_\tau \frac{ (\log \log \eta)^{\frac32} }{\eta \sqrt{\log \eta}},
  \end{align*}
  where $\mu$ is the uniform measure on the Boolean hypercube $\{-1,1\}^n$ and $c_\tau$ is a constant that depends only on $\tau$. This result resolves Talagrand's convolution conjecture up to a dimension-free $(\log \log \eta)^{\frac32}$ factor. Our proof uses the reverse heat process on the Boolean hypercube, a coupling construction with carefully engineered perturbations of jump rates and a time-smoothed anti-concentration estimate. 
\end{abstract}

\section{Introduction}
In 1989, Talagrand put forward a conjecture on the regularization effect of convolution on $L^1$ functions on the Boolean hypercube in Problems 1 and 2 of~\cite{talagrand1989conjecture}. Consider the Boolean hypercube $\braces{-1, 1}^n$. Given $t \geq 0$, let $\xi_t$ be a random vector on $\braces{-1, 1}^n$, whose coordinates $\xii_t$ are independent and identically distributed with
\begin{align*}
  \Prob(\xii_t = 1) = \frac{1+e^{-t}}{2}, \quad \Prob(\xii_t = -1) = \frac{1-e^{-t}}{2}.  
\end{align*}

Consider the \textit{heat semigroup} $(P_t)_{t \geq 0}$ defined for any function $f: \braces{-1, 1}^n \to \real$ by, for any $x \in \braces{-1,1}^n$,
\begin{align}
  \label{eq:heat_semigroup}
  P_t f(x) \defn \Exs\brackets{f(x \odot \xi_t)},
\end{align}
where $\odot$ denotes element-wise product. Let $\mu_t^n$ denote the law of $\xi_t$, then $P_t$ can also be written as a convolution $P_t f(x) = \int f(x \odot y) d\mu_t^n(y) = f \ast \mu_t^n$,
where the convolution operator $\ast$ is defined via the multiplicative group structure on $\braces{-1, 1}^n$. For this reason, the heat semigroup $P_t$ was referred to as ``convolution by a biased coin'' in~\cite{talagrand1989conjecture}. 

Write $\mu \defn \mu_\infty^n$, which is the uniform measure on $\braces{-1, 1}^n$. For $p \geq 1$, $f: \braces{-1, 1}^n \to \real$, define its $L^p(\mu)$-norm to be $\vecnorm{f}{p} \defn \parenth{ \int \abss{f}^p d\mu }^{\frac{1}{p}}$. Talagrand's convolution conjecture, also restated on his website~\cite{talagrand1989prizes}, precisely predicts the regularization effect of the heat semigroup when applied to $L^1(\mu)$ functions, as follows. 
\begin{conjecture}[Problem 2~\cite{talagrand1989conjecture}]
  \label{conj:talagrand}
  For any $\tau > 0$, there exists a constant $c_\tau > 0$ that depends only on $\tau$, such that for every nonnegative function $f: \braces{-1, 1}^n \to \real_+$ with $\vecnorm{f}{1} \neq 0$ and any $\eta > 1$, we have 
  \begin{align*}
    \Prob_{X \sim \mu}(P_\tau f(X) > \eta \vecnorm{f}{1} ) \leq \frac{c_\tau}{\eta \sqrt{\log \eta}}.
  \end{align*}
\end{conjecture}
The regularization effect refers to the gain of a $1/\sqrt{\log \eta}$ factor over Markov's inequality. In Talagrand's words, ``convolution by a biased coin spreads regularity''. In this paper, our main result proves this gain up to a dimension-free $\parenth{\log \log \eta}^{\frac32}$ factor. 

\begin{theorem}
  \label{thm:main}
  There exists a universal constant $c > 0$ such that for any $\tau > 0$, $n \geq 1$, $\eta > e^3$, and $f: \braces{-1, 1}^n \to \real_+$ with $\vecnorm{f}{1} \neq 0$, we have 
  \begin{align*}
    \Prob_{X \sim \mu}(P_\tau f(X) > \eta \vecnorm{f}{1} ) \leq \frac{c}{\parenth{1-e^{-\tau}}^2} \frac{\parenth{ \log \log \eta}^{\frac32} }{\eta \sqrt{\log \eta}}.
  \end{align*}
\end{theorem}
While $1/\sqrt{\log \eta}$ is conjectured to be the extra factor, Talagrand also mentioned in~\cite{talagrand1989prizes} that ``I do not know if it is true that $\lim_{\eta \to \infty} \psi(\eta) = 0$'', where $\psi(\eta) \defn \sup_{n \in \naturalnum} \sup_{f \geq 0} \braces{\eta \Prob(P_\tau f \geq \eta \vecnorm{f}{1})}$. This is the statement of Problem 1 in~\cite{talagrand1989conjecture}. Theorem~\ref{thm:main} is the first dimension-free result which confirms $\lim_{\eta \to \infty} \psi(\eta) = 0$, resolving Talagrand's Problem 1. 

\begin{remark}
    \label{re:thm_main}
    Without loss of generality, by rescaling, we may assume $\vecnorm{f}{1} = 1$. Additionally, it suffices to prove the case for strictly positive functions first, and then obtain the general case by passing to the limit.
\end{remark}
Before we proceed to the main proofs, we review related classical results and summarize prior progress on this conjecture.

\subsection{A few comments on the conjecture}
First, to bound the tail probability of $P_\tau f(X) > \eta \vecnorm{f}{1}$, a natural bound is given by Markov's inequality. Since $\Exs_{X \sim \mu}[P_\tau f(X)] = \Exs_{X\sim \mu}[f(X)] = \vecnorm{f}{1}$, Markov's inequality gives
\begin{align}
  \label{eq:Markov_ineq}
  \Prob_{X\sim \mu} (P_\tau f(X) > \eta \vecnorm{f}{1}) \leq \frac{1}{\eta}. 
\end{align}
This bound holds for any $\tau \geq 0$. Talagrand's conjecture asserts that, as long as $\tau > 0$, this bound can be improved by an additional $1/\sqrt{\log \eta}$ factor.

Second, from Eq.~\eqref{eq:Markov_ineq}, one immediately obtains a bound with an extra $1/\sqrt{\log \eta}$ factor if a dimension-dependent constant is allowed. 
\begin{align*}
  \Prob_{X\sim \mu} (P_\tau f(X) > \eta \vecnorm{f}{1}) \leq \frac{1}{\eta} = \frac{\sqrt{\log \eta}}{\eta \sqrt{\log \eta}} \overset{(i)}{\leq} \frac{\sqrt{n \log 2}}{\eta \sqrt{\log \eta}}. 
\end{align*}
(i) follows from the trivial bound $f \leq 2^n \vecnorm{f}{1}$ (see Lemma~\ref{lem:edge_ratio}), as well as the fact that if $\eta > 2^n$, then the probability on the left-hand side is $0$. Hence, allowing a dimension-dependent constant factor in Conjecture~\ref{conj:talagrand} is not very interesting, unless the constant is much smaller than $\sqrt{n}$.

Third, as shown in Proposition 5 of~\cite{talagrand1989prizes}, if one aims for a uniform dimension-free tail bound which holds for any $f \geq 0$ and any $\eta > 2$, then the best bound one can hope for is $1/(\eta \sqrt{\log \eta})$ up to a constant that depends on $\tau$. Talagrand showed this lower bound in \cite{talagrand1989prizes} by considering the product function $f(x) = \prod_{i=1}^n(1+x_i)$ and a specific threshold $\eta$. Whether this lower bound is tight remains unknown.

Fourth, one may ask why one normalizes by $\vecnorm{\cdot}{1}$ rather than with $\vecnorm{\cdot}{p}$ for $p>1$. As pointed out in~\cite{talagrand1989conjecture}, as long as $p>1$, $P_\tau$ is already known to be a regularizing operator. Specifically, hypercontractivity for the uniform measure on $\braces{-1,1}^n$ implies that, for $ q = 1 + e^{2\tau} (p-1)$, $\vecnorm{P_\tau f}{q} \leq \vecnorm{f}{p}$. Then Markov's inequality gives
\begin{align*}
  \Prob(P_\tau f(X) > \eta \vecnorm{f}{p}) \leq \frac{1}{\eta^q},
\end{align*}
which is better than $1/(\eta \sqrt{\log \eta})$ for large $\eta$ as long as $p > 1$. While hypercontractivity is powerful for $p>1$, it does not yield information about the action of $P_\tau$ on $L^1$ functions. This observation is one main reason that Talagrand formulated the conjecture: to clarify the regularizing properties of $P_\tau$ on $L^1$ functions. 

Finally, as explained in Section 1 of~\cite{eldan2018regularization}, Conjecture~\ref{conj:talagrand} is closely related to isoperimetric inequalities with extra logarithmic factors and to the geometry of small sets via a duality argument.

\subsection{Remarks on the Gaussian counterpart}
In an attempt to attack Conjecture~\ref{conj:talagrand}, a natural Gaussian counterpart of the conjecture, related to the Ornstein-Uhlenbeck (OU) semigroup, was formulated in~\cite{ball2013l1} and has since been studied. The Gaussian results therefore serve as an important point of comparison.

Let $\gamma_n$ denote the standard Gaussian measure in $\real^n$ with density $x \mapsto \frac{1}{(2\pi)^{n/2}} \exp(-\vecnorm{x}{2}^2/2)$, where $\vecnorm{\cdot}{2}$ denotes the Euclidean norm. For $p \geq 1$, let $L^p(\gamma_n) \defn \braces{f: \real^n \to \real \mid \int \abss{f}^p d\gamma_n < \infty}$. Given $f \in L^1(\gamma_n)$, the OU semigroup is defined by Mehler's representation as
\begin{align*}
  P_t^{\text{OU}} f(x) \defn \int f(e^{-t}x + \sqrt{1-e^{-2t}} y) d\gamma_n(y), \quad x \in \real^n, t \geq 0. 
\end{align*}
The Gaussian counterpart of the conjecture asks whether it is true that for every nonnegative function $f: \real^n \to \real_+$ with $\vecnorm{f}{1} \neq 0$, and any $\eta > 1$, we have 
\begin{align}
  \label{eq:Gaussian_conj}
  \Prob_{X \sim \gamma_n}(P^{\text{OU}}_\tau f(X) > \eta \vecnorm{f}{1} ) \leq \frac{c_\tau}{\eta \sqrt{\log \eta}},
\end{align}
for a constant $c_\tau$ that depends only on $\tau > 0$. In words, the Gaussian counterpart replaces the heat semigroup on the Boolean hypercube with the OU semigroup, and the uniform measure on the Boolean hypercube with the standard Gaussian. 
The Gaussian counterpart is fully resolved. Ball, Barthe, Bednorz, Oleszkiewicz and Wolff first formulated this problem in~\cite{ball2013l1} and provided a bound with a constant $c_{\tau, n}$ that depends on the dimension $n$ and an extra $\log \log \eta$ factor in the numerator. Eldan and Lee, using F\"ollmer's process and a coupling construction via stochastic calculus, proved a dimension-free bound up to an extra $\sqrt{\log\log \eta}$ factor in~\cite{eldan2018regularization}. Following the main ideas in \cite{eldan2018regularization}, Lehec refined their stochastic analysis to remove the extra $\sqrt{\log\log \eta}$ factor in~\cite{lehec2016regularization}, fully resolving the Gaussian counterpart.

Both~\cite{eldan2018regularization} and~\cite{lehec2016regularization} crucially use the second-order smoothness (or semi-log-convexity) generated by the OU semigroup, as follows: for any $f:\real^n \to \real_+$ and $t > 0$, we have
\begin{align*}
  \nabla^2 \log P_t^{\text{OU}} f \succeq -\frac{1}{2t} \Ind_n.
\end{align*}
In fact, this semi-log-convexity is the only role of the OU semigroup in resolving the Gaussian counterpart of the conjecture. It is a well-known fact that the Gaussian counterpart of the conjecture follows from Talagrand's Conjecture~\ref{conj:talagrand} by the central limit theorem. On the other hand, whether the Gaussian proof is helpful for proving Conjecture~\ref{conj:talagrand} is unknown.

There are several recent works that extend the success in the Gaussian setting to other measures and/or other semigroups. \cite{gozlan2019deviation} gave an alternative proof of the Gaussian counterpart in dimension $1$.  \cite{gozlan2023log} studied another continuous analogue by replacing the OU semigroup with perturbations of the OU semigroup in dimension $1$, as well as other variants for the $M/M/\infty$ queuing process on the integers and Laguerre semigroup in dimension $1$.

\paragraph{Notation} Let $[n]$ denote the set $\braces{1, 2, \ldots, n}$. $(e_1, e_2, \ldots, e_n)$ denotes the canonical basis of $\real^n$. For $a \in \real$, $(a)_+ = \max\braces{a, 0}$. For $a, b \in \real$, $a \wedge b \defn \min\braces{a, b}$. Let $\odot$ denote the element-wise product. For $x \in \real^n$, $\flip_i(x) \defn (x_1, \ldots, x_{i-1}, -x_i, x_{i+1}, \ldots )$ is the vector obtained by flipping the $i$-th coordinate. For $\phi: \braces{-1,1}^n \to \real$, let $\Delta_i \phi$ denote the ``flip $i$-th coordinate and take difference'' function $x\mapsto \phi(\flip_i(x)) - \phi(x)$. For $h: \braces{-1,1}^{2n} \to \real$, define $\Delta_i^x h$ as $(x, y) \mapsto h(\flip_i(x), y) - h(x, y)$, $\Delta_i^y h$ as $(x, y) \mapsto h(x, \flip_i(y)) - h(x, y)$ and $\Delta_i^{xy} h$ as $(x, y) \mapsto h(\flip_i(x), \flip_i(y)) - h(x, y)$. Denote $\real_+ \defn [0,\infty)$.  For $x\in \real$, 
\begin{align*}
  \sign(x) \defn \begin{cases}
    +1 & \text{ if } x > 0\\
    0 & \text{ if } x = 0\\
    -1 & \text{ otherwise }.
  \end{cases}
\end{align*}
For $x \in \real^n$, $\sign(x)$ is the vector obtained by applying $\sign$ coordinate-wise. For a function $\psi: \real \to \real$, $\supp(\psi)$ denotes the support of $\psi$. We use $\Ent$ to denote the entropy, $\Ent_\mu g = \Exs_\mu[g \log g] - \Exs_\mu[g] \log \Exs_\mu [g]$. We write $a \lesssim b$ if there exists a universal constant $c > 0$ such that $a \leq c b$. The constants $c$ and $c'$ denote universal constants, which may change depending on the context.

\section{Preliminaries}
We introduce a few facts about Boolean function analysis, general continuous-time Markov processes and the heat semigroup on the Boolean hypercube.

\subsection{Boolean function analysis basics}
We recall a few standard concepts without proof. For a detailed exposition on Boolean function analysis, we refer the interested readers to the book~\cite{o2014analysis}. 

\paragraph{Inner product for functions on $\braces{-1,1}^n$.} For a pair of functions $g, h: \braces{-1,1}^n \to \real$, we define their inner product as
\begin{align*}
  \angles{g, h} \defn \int g(x) h(x) d\mu(x) = \Exs_\mu[g h ]. 
\end{align*}

\paragraph{Multilinear expansion of a function on $\braces{-1,1}^n$} For a subset $S \subseteq [n]$, define the monomial corresponding to $S$ as $
  x^S \defn \prod_{i \in S} x_i$, 
with $x^{\emptyset} = 1$ by convention. With the above inner product, $\braces{x^S \mid S \subseteq [n]}$ form an orthonormal basis for the vector space of functions $\braces{-1, 1}^n \to \real$ (see e.g. Theorem 1.5 in~\cite{o2014analysis}). Additionally, every function $g: \braces{-1,1}^n\to \real$ can be uniquely expressed as a multilinear polynomial, 
\begin{align}
  \label{eq:multilinear_expansion}
  g(x) = \sum_{S \subseteq [n]} \widehat{g}(S) x^S,
\end{align} 
where $\widehat{g}(S) \defn \angles{g, x^S}$ is called the Fourier coefficient of $g$ on $S$ (see e.g. Theorem 1.1 in~\cite{o2014analysis}). With the multilinear polynomial representation, $g$ can be identified with a multilinear function on $\real^n$. This also allows us to define the derivative of $g$ in the usual sense, for $i \in [n]$
\begin{align*}
  \partial_i g(x) \defn \lim_{h \to 0}\frac{g(x+h e_i) - g(x)}{h} = \frac{g(\ldots, x_{i-1}, +1, x_{i+1}, \ldots) - g(\ldots, x_{i-1}, -1, x_{i+1}, \ldots)}{2}. 
\end{align*}
Moreover, the zeroth Fourier coefficient satisfies $\hat{g}(\emptyset) = \vecnorm{g}{1} = \int g d\mu = g(0)$ for nonnegative $g$. 

\paragraph{Explicit form of the heat semigroup} Using the multilinear expansion, the heat semigroup~\eqref{eq:heat_semigroup} can be written as
\begin{align}
  \label{eq:explicit_heat_seimgroup}
  P_t g (x) &= \int g(x \odot y) d\mu_t^n(y) = \sum_{y \in \braces{-1,1}^n} \sum_{S \subseteq [n]} \hat{g}(S) x^S y^S \prod_{i \in [n]} \parenth{\frac{1+y_i e^{-t}}{2}} \overset{(i)}{=} g(e^{-t} x).
\end{align}
(i) follows by summing over $y$ first. Vaguely speaking, applying the heat semigroup ``moves'' the input space from $\braces{-1,1}^n$ to the inner hypercube $\braces{-e^{-t}, e^{-t}}^n$, which may make the multilinear polynomial more regular.

\subsection{Continuous-time Markov processes}
\label{sub:continuous_time_Markov}
In this subsection, we introduce a few concepts related to Markov processes and Markov semigroups. Interested readers are referred to \cite{anderson2012continuous} for continuous-time Markov chains and~\cite{pazy2012semigroups} for a semigroup viewpoint for additional details. All Markov processes considered are c\`adl\`ag (right continuous with left limits).

Given a Markov process $(X_t)_{t \geq 0}$, taking values in a finite space $\Omega$.

\paragraph{Homogeneous case} When $(X_t)_{t \geq 0}$ is time-homogeneous (i.e., the law of $X_t \mid X_s = x$ stays the same as $X_{t-s} \mid X_0 = x$, $\forall t\geq s \geq 0$), we define the associated \textit{Markov semigroup} $Q_t$ as 
\begin{align*}
  Q_0 &= \Ind \\
  Q_t h(x) &= \Exs[h(X_t) \mid X_0 = x],
\end{align*}
for every bounded function $h: \Omega \to \real$, and every $x \in \Omega$. It also acts on measures through the duality $\int h d (\nu Q_t) = \int Q_t h d\nu$. Its \textit{generator} is defined as
\begin{align*}
  L h (x) \defn \lim_{t \to 0^+} \frac{Q_t h (x) - h(x)}{t},
\end{align*}
provided the limit exists. 
A measure $\nu$ on $\Omega$ is said to be \textit{invariant} for $(Q_t)$ if $\int Q_t h d\nu = \int h d \nu, \forall h, \forall t \geq 0$. The Markov property guarantees that it satisfies the semigroup property, namely, $Q_{t+s} = Q_t Q_s, \forall 0 \leq s \leq t$.

\paragraph{Inhomogeneous case} When $X_t$ is time-inhomogeneous, we work with a two-parameter semigroup (also called evolution system) $(Q_{s, t})_{0 \leq s \leq t}$ acting on bounded functions $\Omega \to \real$, such that
\begin{align*}
  Q_{s, s} &= \Ind \\
  Q_{s, t} h(x) &= \Exs[h(X_t) \mid X_s = x].  
\end{align*}
The Markov property guarantees that it satisfies the two-parameter semigroup property $Q_{s,t} = Q_{s,r}Q_{r,t}, 0 \leq s\leq r \leq t$. Additionally, on a finite space $\Omega$, $Q_{s, t}$ can be viewed as a matrix of size $\abss{\Omega} \times \abss{\Omega}$. 
Its time-dependent generator $L_t$ is defined via the Kolmogorov forward and backward equations (if the partial derivative with respect to $t$ is well-defined): 
\begin{align*}
  \partial_t Q_{s, t} = Q_{s, t} L_t , \quad \partial_s Q_{s, t}  = - L_s Q_{s,t}.
\end{align*}

\subsection{Heat semigroup and its associated Markov process}
In the heat semigroup definition in Eq.~\eqref{eq:heat_semigroup}, $\xi_t$ only specifies one-time marginals of the process. We introduce a Markov process which fulfills the one-time marginals via a jump stochastic differential equation (SDE), making the name ``heat semigroup'' consistent with our semigroup definition in Section~\ref{sub:continuous_time_Markov}. 

Let $f$ be a function $\braces{-1,1}^n \to (0, \infty)$ with $\vecnorm{f}{1} = 1$. Let $E \defn \real^2$, equipped with Lebesgue measure. Consider a filtered probability space $\parenth{\Omega, \sFU, \parenth{\sFU_t}_{t \geq 0}, \Prob}$ carrying the following objects.
Let $\fU_0 \sim \nu_f$, where $\nu_f \defn f \cdot \mu$ is a probability measure with density $f$ with respect to the uniform measure $\mu$ on $\braces{-1, 1}^n$. Independent of $\fU_0$, let $N^U$ be a Poisson random measure (PRM) on $\real_+ \times E$ with intensity $dt dz_1 dz_2$. Let $(\sFU_t)_{t \geq 0}$ be the natural filtration generated by $\fU_0$ and $N^U$. Let $\widetilde{N}^U$ be the associated compensated PRM. For $i \in [n]$, define the coordinate PRM
\begin{align*}
  \NUi(dt, dz) \defn N^U(dt,  (i, i+1] \times dz).
\end{align*}
Then $N^{U, [1]}, \ldots, N^{U, [n]}$ are $n$ independent PRMs on $\real_+ \times \real$ with intensity $dt dz$. Their compensated versions are $\widetilde{N}^{U, [1]}, \ldots, \widetilde{N}^{U, [n]}$.

\paragraph{Forward jump process} We define the forward jump process through the following jump SDE,
\begin{align}
  \label{eq:def_forward_process}
  \fU_0 &\sim \nu_f = f \cdot \mu, \notag \\
  d\fU_t &\defn \sum_{i=1}^n (-2 \fU_{t-} e_i) \int_{\real}\mathbf{1}_{0 < z < \frac12} \NUi(dt, dz).
\end{align}

For any $T > 0$, existence and uniqueness of the solution of the SDE on $[0, T]$ follow from the bounded total jump rate for pure-jump processes (see e.g., Chapter 4.7~\cite{ethier2009markov}). Since each coordinate of $\fU_t$ evolves independently, the distribution of $\fU_t$ is determined by its one-coordinate distributions. The first coordinate flips according to a Poisson process of rate $1/2$. Its conditional mean has an explicit form, after solving an ordinary differential equation (ODE), 
\begin{align*}
  \Exs\brackets{\fU_t^\tagone \mid \fU_0^\tagone = x} = x e^{-t}. 
\end{align*}
Then the transition probability follows since $\fU_t^\tagone$ only takes value in $\braces{-1,1}$: $\Prob(\fU_t^\tagone = 1 \mid \fU_0^\tagone = x) = \frac{1+xe^{-t}}{2}$. In other words, $\fU_t \mid \fU_0 = x$ matches the distribution of $x \odot \xi_t$. Hence, the semigroup associated to $(\fU_t)_{t\geq 0}$ is exactly the heat semigroup $(P_t)_{t\geq 0}$ defined in Eq.~\eqref{eq:heat_semigroup}. 

Next, we describe its generator and its marginal law. Its generator takes the form, for any test function $h: \braces{-1,1}^n \to \real$,
\begin{align}
  \label{eq:heat_semigroup_generator}
  \fL h(x) = \frac{1}{2} \sum_{i=1}^n \Delta_i h(x). 
\end{align}
In matrix form, $\fL(x, y) = \frac12 \sum_{i=1}^n \mathbf{1}_{y = \flip_i(x)}$, \text{ for } $x \neq y \in \braces{-1, 1}^n$. Additionally, the uniform measure $\mu$ is the invariant measure for $(P_t)_{t\geq 0}$. For time $t > 0$, the one-time marginal of $\fU_t$ is $p_t \defn \nu_f P_t$, which has an explicit form $p_t(x) = f(e^{-t} x) \mu(x)$. Conceptually, as $t$ goes from $0$ to $\infty$, one may think of the process $(\fU_t)_{t\geq 0}$ as creating a stochastic bridge from the probability measure $\nu_f$ to the uniform measure $\mu$.

\subsection{Time reversal of a Markov process}
Given a Markov process $(X_t)$ on a finite space $\Omega$ and its associated semigroup $(Q_{s, t})_{0 \leq s \leq t}$ with generator $L_t$. Let $\pi_t = \pi_0 Q_{0, t}$ be a family of laws of $X_t$. The \textit{time-reversed semigroup} (with respect to $(\pi_t)$) is defined as
\begin{align*}
  \widetilde{Q}_{s,t}(y, x) \defn \frac{\pi_s(x)}{\pi_t(y)} Q_{s, t} (x, y), \forall s \leq t, \forall x, y \in \Omega.
\end{align*}
Its reversed generator satisfies the \textit{flux equation}
\begin{align*}
  \tilde{L}_t(y, x) &= \frac{\pi_t(x)}{\pi_t(y)} L_t(x, y), \text{ if } x \neq y, \\
  \tilde{L}_t(y, y) &= - \sum_{x \neq y} \tilde{L}_t(y, x).
\end{align*}

The time reversal has a pathwise interpretation. Define $\widetilde{X}_{t} \defn X_{T-t}$ for $t \in [0, T]$. Then $\widetilde{X}_t$ is a Markov process with initial law $\pi_T$, with $\parenth{\widetilde{Q}_{T-t,T-s}}$ as its associated semigroup and $\parenth{\tilde{L}_{T-t}}$ as its generator. This pathwise interpretation identifies the law of the reversed process. To define the reverse process properly, we need to realize it on a reverse filtered probability space with its own PRMs. See e.g., Chapter 7~\cite{anderson2012continuous} for more details about time-reversed Markov processes. 

\subsection{Time reversal of the heat process}
Using the time reversal formula, we identify the law of the time reversal of the forward heat process in Eq.~\eqref{eq:def_forward_process}. The time reversal of the forward process over the interval $[0, T]$ has initial law $\nu_{P_T f}$. We can realize it on a reverse filtered probability space 
\begin{align*}
  \parenth{\Omega, \sF, \parenth{\sF_t}_{t \in [0, T]}, \Prob}
\end{align*}
carrying the following objects. Let $\uV_0 \sim \nu_{P_T f}$. Independent of $\uV_0$, let $N^{[1]}, \ldots, N^{[n]}$ be $n$ independent PRMs on $\real_+ \times \real$ with intensity $dt dz$. Let $(\sF_t)_{t \in [0, T]}$ be the natural filtration generated by $\uV_0$ and these PRMs. Let $\wNi$ be the associated compensated PRMs. From this point onward, unless explicitly stated otherwise, all stochastic objects are defined on this reverse filtered probability space.

\paragraph{Reverse heat process} Define $(\uV_t)_{t \in [0, T]}$ as the reverse heat process, which satisfies the following SDE
\begin{align}
  \label{eq:def_reverse_process}
  \uV_0 &\sim \nu_{P_T f}, \notag \\
  d \uV_t &=  \sum_{i=1}^n (-2 \uV_{t-} e_i) \int_\real \mathbf{1}_{0 < z \leq \frac12 - S_i(e^{-{(T-t)}} \uV_{t-})} N^\tagi(dt, dz),
\end{align}
where the \textit{score function} for the $i$-th coordinate $S_i: [-1,1]^n \to \real$, when $f$ is identified with its multilinear expansion, is defined as
\begin{align}
  \label{eq:def_score_fun}
  S_i(x) \defn \frac{x_i \partial_i f (x)}{ f(x)}, \quad \forall i \in [n].
\end{align}
The reverse jump rate is derived as the unique solution to the flux equation. 
Recall that for the heat semigroup, the marginal law is $p_t(x) = f(e^{-t} x) \mu(x)$ and $\fL(x, y) = \frac12 \sum_{i=1}^n \mathbf{1}_{y = \flip_i(x)}$. The flux equation provides an explicit form of the generator, for any function $h:\braces{-1,1}^n \to \real$,
\begin{align*}
  \LV_t h(x) = \frac12 \sum_{i=1}^n \frac{f(\flip_i(e^{-(T-t)} x))}{f(e^{-(T-t)} x)}  \Delta_i h(x).
\end{align*}
Since $f > 0$ and the space $\braces{-1,1}^n$ is finite, the jump rate is finite on $[0, T]$. By construction, the law of $(\uV_t)_{t \in [0, T]}$ is the same as the law of $(\fU_{T-t})_{t \in [0, T]}$ from the forward process. Existence and uniqueness of the solution of the SDE follow from the bounded total jump rate for pure-jump processes (see e.g., Chapter 4.7~\cite{ethier2009markov}).

Unlike the forward process, the reverse heat process is time-inhomogeneous due to its time-dependent jump rates. In view of the jump rate, it is natural to introduce the \textit{scaled reverse process}
\begin{align}
  \label{eq:def_scaled_reverse_jump_process}
  \rV_t \defn e^{-(T-t)} \uV_t, 
\end{align}
which lives in the inner cube $[-1,1]^n$ and becomes a time-homogeneous Markov process. 
\begin{remark}
  In light of the recent trendy development on score-based diffusion models for sampling high-dimensional natural images~\cite{song2021score} via time-reversal of Ornstein-Uhlenbeck process, we note that $(\uV_t)_{t \in [0, T]}$ is the Boolean hypercube analogue of the Gaussian diffusion models. Here, the Gaussian noise is replaced by Poisson jumps. 

  Note that, as a vector, $S(\rV_t) = \uV_t \odot \nabla \log P_{T-t} f (\uV_t)$. Its resemblance to F\"ollmer's drift or the score function in the Gaussian setting~\cite{eldan2018regularization} motivates us to name it the ``score function'', too. It is known that F\"ollmer's drift is a martingale in the Gaussian setting (see e.g. Section 2~\cite{eldan2018regularization}). The score $S(\rV_t)$, in our definition, is also a martingale, which we show in the Appendix (Lemma~\ref{lem:score_martingale}).
\end{remark}

\section{Main proof}
\label{sec:main_proof}
Fix $\tau > 0$, $T > \tau + 1$, $\To \defn T - \tau$ and $\btau = \frac{1+e^{-\tau}}{1-e^{-\tau}}$. For $t \in [0, T]$, let $\rho_t = e^{-(T-t)}$. By Remark~\ref{re:thm_main}, we restrict ourselves to a strictly positive function $f: \braces{-1, 1}^n \to (0, \infty)$ with $\vecnorm{f}{1} = 1$ and fix it throughout. To prove Theorem~\ref{thm:main}, the high-level proof strategy is to construct a coupling to prove an anti-concentration bound, which is inspired by the strategy in the Gaussian counterpart~\cite{eldan2018regularization}. However, the coupling construction and technical details differ significantly due to the discrete nature of the Boolean hypercube $\braces{-1,1}^n$. For an interval $I \subseteq \real$, define the anti-concentration profile
\begin{align}
  \label{eq:def_anti_concentration_profile}
  \fA_t(I) \defn \Prob_{Y \sim \nu_{P_t f} } \parenth{\log P_t f(Y) \in I},
\end{align}
where $\nu_{P_t f}$ is the probability measure on $\braces{-1,1}^n$ with density $P_t f$ with respect to the uniform measure $\mu$.
The first step is to reduce the tail bound in Theorem~\ref{thm:main} to a bound on the anti-concentration profile.

\begin{lemma}[Anti-concentration profile]
  \label{lem:anti_concentration}
  For a strictly positive function $f : \{-1,1\}^n \to (0,\infty)$ with $\vecnorm{f}{1} = 1$, for any $\tau > 0$, there exists a universal constant $c > 0$, such that for any $\eta > e^{8}$, 
  \begin{align*}
    \fA_\tau((\log \eta, \log(\eta) + 1]) 
    &\leq \frac{c}{\parenth{1-e^{-\tau}}^2} \frac{\parenth{\log \log \eta}^{\frac32}}{\sqrt{\log \eta}}.
  \end{align*}
\end{lemma}

Theorem~\ref{thm:main} directly follows from Lemma~\ref{lem:anti_concentration}, using the same argument as in the Gaussian space (see Section 2~\cite{eldan2018regularization}). A short proof is provided for completeness. 
\begin{proof}[Proof of Theorem~\ref{thm:main}]
  Let $c_\tau = \frac{c}{\parenth{1-e^{-\tau}}^2}$ for $c$ in Lemma~\ref{lem:anti_concentration}. For $\eta > e^8$, we have 
  \begin{align*}
    \Prob_{X\sim \mu}(P_\tau f (X) > \eta) &= \sum_{k=0}^\infty \Prob_{X\sim \mu}(P_\tau f (X) \in (e^k \eta, e^{k+1} \eta]) \\
    &\leq \sum_{k=0}^\infty \frac{1}{e^k \eta} \Exs _{X\sim \mu} \brackets{P_\tau f (X) \mathbf{1}_{ P_\tau f (X) \in (e^k \eta, e^{k+1} \eta]} } \\
    &= \sum_{k=0}^\infty \frac{1}{e^k \eta} \fA_\tau((\log (e^{k} \eta), \log(e^k \eta) + 1]) \\
    &\overset{(i)}{\leq} \sum_{k=0}^\infty \frac{1}{e^k \eta} c_\tau \frac{\parenth{\log \log (e^k \eta)}^{\frac32}}{\sqrt{\log (e^k\eta)}} \\
    &\overset{(ii)}{\leq}  2 c_\tau \frac{ \parenth{\log \log (\eta)}^{\frac32} }{\eta \sqrt{\log (\eta)}} \sum_{k=0}^\infty \frac{1}{e^k} \\
    &=  \frac{2 c_\tau e}{e-1} \frac{ \parenth{\log \log (\eta)}^{\frac32} }{\eta \sqrt{\log (\eta)}}.
  \end{align*}
  (i) follows from Lemma~\ref{lem:anti_concentration}. Let $h: x\mapsto \parenth{\log \log (x)}^{\frac32}/ \sqrt{\log (x)}$. (ii) follows from the fact that $h(\eta) \geq \frac{1}{2} \max_{x > \eta} h$ for $\eta > e^8$. Since we only care about the result up to universal constants, for $e^3 < \eta \leq e^{8}$, it suffices to use the bound from Markov's inequality. To deal with general nonnegative functions with $\vecnorm{f}{1} \neq 0$, we can apply the same proof to $(f + \varepsilon)/(\vecnorm{f}{1} + \varepsilon)$ for $\varepsilon > 0$ and then take the limit $\varepsilon \to 0^+$.
\end{proof}

The second step is to construct a coupling via the reverse heat process to prove Lemma~\ref{lem:anti_concentration}. We seek a coupling of two processes $(\uV_t, \uW_t)$ such that, marginally at time $\To$, $\Law(\uV_{\To})$ (which is $\nu_{P_\tau f}$) and $\Law(\uW_{\To})$ satisfy the following two properties:
\begin{enumerate}
    \item (TV distance control) $\Law(\uV_{\To})$ and $\Law(\uW_{\To})$ are close in total variation (TV);
    \item (Approximate monotonicity at tail) with high probability, $\log P_\tau f(\uV_{\To})$ is strictly larger than $\log P_\tau f(\uW_{\To})$ by $1$, given that $P_\tau f(\uW_{\To}) > \eta$. 
\end{enumerate}
Combining the above two properties, we would obtain a bound on the anti-concentration profile $\fA_\tau((\log \eta, \log(\eta) + 1])$. While this coupling idea is inspired by the Gaussian proof~\cite{eldan2018regularization}, there are two main challenges in the Boolean setting. First, the design of the coupling is more involved as we cannot perturb the drift directly as before because such a perturbation would leave the Boolean hypercube $\braces{-1,1}^n$. Second, even if such a coupling construction is obtained, it is not clear how to control the TV distance and the approximate monotonicity at tail at the same time in a dimension-free way as Boolean specific obstructions arise.

The main novelty of our proof lies in a careful design of the coupling construction and an analysis of its properties beyond what is achievable with standard Gaussian techniques. 
Specifically, in addition to the TV distance control and the approximate monotonicity at tail, we introduce a time-smoothed anti-concentration profile bound, which is crucial for obtaining a dimension-free bound on the anti-concentration profile. We show that these three properties together are sufficient to prove Lemma~\ref{lem:anti_concentration}. 

Let $\theta \in [0, \To)$ be the starting time of the perturbation in our coupling construction. Let $\excess \defn \frac{1}{2} \log \log \eta + 1 > 0$, $R_\theta \defn [\log \eta - \log P_{T-\theta} f(\uV_\theta)]_+$ be the remaining gap at time $\theta$, and $I_\theta \defn \Exs\brackets{\frac{\mathbf{1}_{R_\theta \geq \excess}}{R_\theta + 1}}$ be the average inverse remaining gap at time $\theta$ when the remaining gap is larger than $\excess$, where the expectation is taken with respect to $\uV_\theta$.

\begin{lemma}[TV distance control]
  \label{lem:coupling_TV_distance_control}
  Let $\eta > e^3$. For a strictly positive function $f : \{-1,1\}^n \to (0,\infty)$ with $\vecnorm{f}{1} = 1$, for $(\uV_t)_{t \in [0, T]}$ defined in Eq.~\eqref{eq:def_reverse_process}, there exists a coupling construction $(\uV_t, \uW_t)_{t \in [\theta, T]}$, described in Section~\ref{sub:main_stochastic_construction}, such that
  \begin{align*}
    \tvdist{\Law(\uV_{\To})}{\Law(\uW_{\To})} \lesssim \btau^2 \excess^2 I_\theta + \sqrt{\btau^2 \excess^2 I_\theta} \sqrt{1 + \frac{e^{-2\tau}}{1-e^{-2\tau}} (\To - \theta)}.
  \end{align*}
\end{lemma}
Lemma~\ref{lem:coupling_TV_distance_control} ensures that the two random variables $\uV_{\To}$ and $\uW_{\To}$ are close in TV distance by the coupling construction. Its proof is in Section~\ref{sub:proof_lem:coupling_TV_distance_control}.

\begin{lemma}[Approximate monotonicity at tail]
  \label{lem:approx_monotone_coupling}
  Let $\eta > e^3$. The same coupling construction in Lemma~\ref{lem:coupling_TV_distance_control} satisfies
  \begin{align*}
    \Prob\parenth{\log P_{\tau}f(\uV_\To) > \log \eta + 1}
    \geq \Prob\parenth{\log P_\tau f(\uW_\To) > \log \eta}  - \fA_{T-\theta}((\log \eta - \excess, \log \eta + \excess]) - \frac{3}{\sqrt{\log \eta}}.
  \end{align*}
\end{lemma}
Lemma~\ref{lem:approx_monotone_coupling} shows that $\log P_\tau f(\uV_\To)$ is slightly larger than $\log P_\tau f(\uW_\To)$ at the tail level $\log \eta$ by the coupling construction. Its proof is in Section~\ref{sub:proof_lem:approx_monotone_coupling}.

\begin{lemma}[Time-smoothed anti-concentration profile bound]
  \label{lem:anticoncentration_profile_bound}
  For a strictly positive function $f : \{-1,1\}^n \to (0,\infty)$ with $\vecnorm{f}{1} = 1$, let $\fA_t$ be the anti-concentration profile defined in Eq.~\eqref{eq:def_anti_concentration_profile}. 
  For any $\ell > 2$, we have
  \begin{align*}
    \int_{0}^{\infty} \fA_s((\ell, \ell+1]) ds \lesssim \frac{1}{\ell}.
  \end{align*}
  
\end{lemma}
Lemma~\ref{lem:anticoncentration_profile_bound} shows that when one smooths over time, the anti-concentration profile is controlled at a $1/\ell$ rate, which is much better than the desired fixed-time rate of $1/\sqrt{\ell}$ in Lemma~\ref{lem:anti_concentration}. Lemma~\ref{lem:anticoncentration_profile_bound} strengthens the anti-concentration bound obtained from a single-start coupling construction with Lemma~\ref{lem:coupling_TV_distance_control} and Lemma~\ref{lem:approx_monotone_coupling} alone. Its proof is in Section~\ref{sub:time_smoothed_anti_concentration_profile_bound}.

Given the three lemmas above, we are ready to prove Lemma~\ref{lem:anti_concentration}. 
\begin{proof}[Proof of Lemma~\ref{lem:anti_concentration}]
  Let $\eta > e^8$. Then $\log \eta - \frac12 \log \log \eta - 1 > \frac12 \log \eta \geq 2$ and $\excess = \frac12 \log \log \eta + 1 > 2$. Combining Lemma~\ref{lem:coupling_TV_distance_control} and Lemma~\ref{lem:approx_monotone_coupling}, we have
  \begin{align}
    \label{eq:anti_concentration_bound_bootstrap}
    &\quad \fA_\tau((\log \eta, \log(\eta) + 1]) \notag \\
    &= \Prob\parenth{\log P_{\tau} f(\uV_\To) > \log \eta }  -  \Prob\parenth{\log P_{\tau} f(\uV_\To) > \log \eta + 1} \notag \\
    &= \brackets{\Prob\parenth{\log P_{\tau} f(\uV_\To) > \log \eta } - \Prob\parenth{\log P_{\tau} f(\uW_\To) > \log \eta }} \notag \\
    &\quad \quad + \brackets{\Prob\parenth{\log P_{\tau} f(\uW_\To) > \log \eta } -  \Prob\parenth{\log P_{\tau} f(\uV_\To) > \log \eta + 1}} \notag \\
    &\lesssim \btau^2 \excess^2 I_\theta + \sqrt{\btau^2 \excess^2 I_\theta} \sqrt{1 + \frac{e^{-2\tau}}{1-e^{-2\tau}} (\To - \theta)} + \fA_{T-\theta}((\log \eta - \excess, \log \eta + \excess])  + \frac{3}{\sqrt{\log \eta}},
  \end{align}

  Apply the above construction for all $\theta \in (\To - 1, \To)$, take the average and apply Jensen's inequality for the square root function, we obtain
  \begin{align}
    \label{eq:anti_concentration_profile_bound_in_proof}
    \fA_\tau((\log \eta, \log(\eta) + 1]) &\lesssim \btau^2 \excess^2 \tilde{I} + \sqrt{\btau^2 \excess^2 \tilde{I}} \sqrt{\btau} + \int_{\To - 1}^{\To} \fA_{T-\theta}((\log \eta - \excess, \log \eta + \excess]) d\theta + \frac{1}{\sqrt{\log \eta}} \notag \\
    &\lesssim \btau^2 \excess^2 \tilde{I} + \sqrt{\btau^3 \excess^2 \tilde{I}} + \frac{\excess}{\log \eta} + \frac{1}{\sqrt{\log \eta}},
  \end{align}
  where $\tilde{I} \defn \int_{\To - 1}^{\To} I_\theta d\theta$ and the last line follows from Lemma~\ref{lem:anticoncentration_profile_bound} by covering $(\log \eta - \excess, \log \eta + \excess]$ with at most $\ceils{2\excess} + 1$ unit intervals and using $\excess = \frac12 \log \log \eta + 1$. 
  
  It remains to bound $\tilde{I}$. By definition, 
  \begin{align*}
    \tilde{I} = \int_{\To - 1}^{\To} I_\theta d\theta = \int_{\To - 1}^{\To} \Exs\brackets{\frac{\mathbf{1}_{R_\theta \geq \excess}}{R_\theta + 1}} d\theta,
  \end{align*}
  where $R_\theta = [\log \eta - \log P_{T-\theta} f(\uV_\theta)]_+$. We write 
  \begin{align*}
    \mathbf{1}_{R_\theta \geq \excess} = \mathbf{1}_{R_\theta > \frac12 \log \eta} + \mathbf{1}_{R_\theta \in [\excess, \frac12 \log \eta]}.
  \end{align*}
  For the first term, we have
  \begin{align}
    \label{eq:I_bound_first_term}
    \int_{\To - 1}^{\To} \Exs\brackets{\frac{\mathbf{1}_{R_\theta \geq \frac12 \log \eta}}{R_\theta + 1}} d\theta \leq \int_{\To - 1}^{\To} \frac{1}{\frac12 \log \eta + 1} d\theta \lesssim \frac{1}{\log \eta}.
  \end{align}
  For the second term, we decompose into unit levels and apply the definition of anti-concentration profile,
  \begin{align*}
    \Exs\brackets{\frac{\mathbf{1}_{R_\theta \in [\excess, \frac12 \log \eta]}}{R_\theta + 1}} &\leq \sum_{r = \floors{\excess}}^{\ceils{\frac12 \log \eta}} \frac{1}{r+1} \Prob(R_\theta \in [r, r + 1]) \\
    &\leq \sum_{r = \floors{\excess}}^{\ceils{\frac12 \log \eta}} \frac{1}{r+1} \fA_{T-\theta}((\log \eta - r - 1, \log \eta - r]).
  \end{align*}
  Integrating from $\To - 1$ to $\To$ and applying Lemma~\ref{lem:anticoncentration_profile_bound}, we obtain
  \begin{align}
    \label{eq:I_bound_second_term}
    \int_{\To - 1}^{\To} \Exs\brackets{\frac{\mathbf{1}_{R_\theta \in [\excess, \frac12 \log \eta]}}{R_\theta + 1}} d\theta &\lesssim \frac{1}{\log \eta} \sum_{r = \floors{\excess}}^{\ceils{\frac12 \log \eta}} \frac{1}{r+1}  \notag \\
    &\lesssim \frac{\excess}{\log \eta}.
  \end{align}
  Combining Eq.~\eqref{eq:I_bound_first_term} and Eq.~\eqref{eq:I_bound_second_term}, we obtain $\tilde{I} \lesssim \frac{\excess}{\log \eta}$. Plugging the $\tilde{I}$ bound back to Eq.~\eqref{eq:anti_concentration_profile_bound_in_proof}, we conclude that
  \begin{align*}
    \fA_\tau((\log \eta, \log(\eta) + 1]) \lesssim \frac{\btau^2 \excess^{\frac32}}{\sqrt{\log \eta}}.
  \end{align*}

\end{proof}
\begin{remark}
  Eq.~\eqref{eq:anti_concentration_bound_bootstrap} is the bootstrap inequality for anti-concentration profiles at two different times. In view of the Gaussian proof in~\cite{eldan2018regularization}, it is tempting to choose $T \approx \log n$ and a fixed starting time $\theta = 0$ in the application of Eq.~\eqref{eq:anti_concentration_bound_bootstrap}. With this choice, both the term $\frac{1}{R_\theta + 1}$ in Lemma~\ref{lem:coupling_TV_distance_control} and the term $\fA_{T-\theta}$ in Lemma~\ref{lem:approx_monotone_coupling} become easier to bound because $\log P_T f(\uV_0) \leq \log (1+e^{-T})^n \leq 1$. However, the perturbation duration term $\sqrt{\To - \theta}$ in Lemma~\ref{lem:coupling_TV_distance_control} causes $\sqrt{\log n}$ in the final bound, which is undesirable. It is worth mentioning that such a dimension-dependent term does not appear if we follow the same proof strategy in the Gaussian setting. 
  
  The key novelty in the Boolean setting is to first notice that the time-smoothed anti-concentration profile has a much better bound $1/\ell$ instead of $1/\sqrt{\ell}$ as shown in Lemma~\ref{lem:anticoncentration_profile_bound}. Then we can apply the coupling construction for all $\theta \in (\To - 1, \To)$ and take the time average. This allows us to prevent the term $\To - \theta$ from being dimension-dependent while still being able to control $\frac{1}{R_\theta + 1}$ and $\fA_{T-\theta}$ terms to be small on average in Lemma~\ref{lem:coupling_TV_distance_control}.

  Lemma~\ref{lem:anticoncentration_profile_bound} is reminiscent of Theorem 1 in~\cite{talagrand1989conjecture} about the tail bound of the time-smoothed heat semigroup $\int_0^1 P_t dt$. While the tail bound of $P_\tau$ is conjectured to be of order $1/(\eta \sqrt{\log \eta})$, Talagrand shows that the tail bound of the time-smoothed heat semigroup $\int_0^1 P_t dt$ is $\log\log \eta / (\eta \log \eta)$, which is much better. Here, we deal with the anti-concentration profile instead of the tail bound, but the improvement from $1/\sqrt{\ell}$ to $1/\ell$ with time-smoothing is of the same flavor. The proof of Lemma~\ref{lem:anticoncentration_profile_bound} is also inspired by the proof of Theorem 1 in~\cite{talagrand1989conjecture}, but technical details are different. 
  \end{remark}

The rest of this section is organized as follows. In Section~\ref{sub:main_stochastic_construction}, we introduce our coupling construction.  Section~\ref{sub:proof_lem:coupling_TV_distance_control} and Section~\ref{sub:proof_lem:approx_monotone_coupling} prove Lemma~\ref{lem:coupling_TV_distance_control} and Lemma~\ref{lem:approx_monotone_coupling} respectively. In Section~\ref{sub:time_smoothed_anti_concentration_profile_bound}, we prove Lemma~\ref{lem:anticoncentration_profile_bound}.

\subsection{Main coupling construction}
\label{sub:main_stochastic_construction}
The F\"ollmer process~\cite{follmer2005entropy}, as the (time-reparametrized) time-reversed Ornstein-Uhlenbeck (OU) process, plays an important role 
in constructing a coupling in the Gaussian counterpart of Talagrand's conjecture~\cite{eldan2018regularization}. Following this idea, we perturb the reverse heat process~\eqref{eq:def_reverse_process} to construct a coupling. The perturbation no longer applies to the drift, but to the jump rate. Recall that the reverse heat process $(\uV_t)_{t\in [0, T]}$ satisfies the following SDE
\begin{align*}
  \uV_0 &\sim \nu_{P_T f}, \notag\\
  d \uV_t &=  \sum_{i=1}^n (-2 \uV_{t-} e_i) \int_\real \mathbf{1}_{0 < z \leq \frac12 - S_i(e^{-{(T-t)}} \uV_{t-})} N^\tagi(dt, dz),
\end{align*}
where the score function is $S_i(x) = \frac{x_i \partial_i f (x)}{ f(x)}$. 

Recall that $\theta \in [0, \To)$ is the starting time of the perturbation,  $\excess \defn \frac{1}{2} \log \log \eta + 1 > 0$, and $R_\theta \defn [\log \eta - \log P_{T-\theta} f(\uV_\theta)]_+$ is the remaining gap at time $\theta$.
\paragraph{Perturbed reverse heat process}  We define for $t\in [\theta, T]$,
\begin{align}
  \label{eq:def_SDE_perturbed_reverse_process}
  \uW_\theta &\defn \uV_\theta \notag \\
  d \uW_t &\defn \sum_{i=1}^n (-2 \uW_{t-} e_i) \int_\real \mathbf{1}_{0 < z \leq \frac12 - (1-\delta_i  \mathbf{1}_{t \leq \stopT}) S_i(e^{-(T-t)}\uV_{t-})} \Ni(dt, dz),
\end{align}
where the $(\sF_t)$-stopping time is defined as
\begin{align}
  \label{eq:def_stopT}
  \stopT \defn \inf \braces{t \in [\theta, \To] \mid \max\braces{\log P_{T-t}f(\uV_t) - \excess,\ \log P_{T-t}f(\uW_t) } \geq \log \eta} \wedge \To,
\end{align}
and $\delta_i$ is a shorthand for $\delta_i (e^{-(T-t)}\uV_{t-})$, where 
\begin{align}
  \label{eq:def_delta_i}
  \delta_i(x) \defn  \bdelta  \brackets{\mathbf{1}_{S_i(x) > 0} + \frac{1-2 S_i(x)}{1 - 2 \bdelta S_i(x)} \mathbf{1}_{S_i(x) \leq 0}}, \quad i \in [n],\quad x\in [-1,1]^n.
\end{align}
where $\bdelta \defn \frac{\excess \mathbf{1}_{R_\theta \geq \excess}}{R_\theta + 1} \in [0, 1)$ is the perturbation size. Since $R_\theta$ is $\sF_\theta$-measurable, the perturbation size $\bdelta$ is also $\sF_\theta$-measurable and thus frozen after time $\theta$.
In words, the stopping time $\stopT$ is the first time at which either $\log P_{T-t}f(\uW_t)$ is at least $ \log \eta$ or $\log P_{T-t}f(\uV_t)$ is at least $\log \eta + \excess$, with the convention that $\stopT$ takes value $\To$ if neither excess is reached. The size of the perturbation is inversely proportional to the remaining gap. Additionally, it has two forms depending on whether the score $S_i$ is positive or not, to make the approximate monotonicity in Lemma~\ref{lem:approx_monotone_coupling} work.

\begin{remark}
  In Eldan and Lee's proof of the Gaussian counterpart~\cite{eldan2018regularization}, a $\delta$ perturbation is added to the score function, which results in a $\delta$ perturbation of the drift of F\"ollmer's process. Here, we cannot add perturbations to the drift as it would leave $\braces{-1,1}^n$. Nevertheless, the idea of perturbing the score, once the score on the Boolean hypercube is properly defined, remains valid and essential.
  
  Moreover, unlike in the Gaussian case, a constant $\delta$ perturbation is not sufficient here. Our coupling construction takes full advantage of the perturbation along the stochastic process: the perturbation $\delta_i$ is state-dependent and coordinate-dependent. 
\end{remark}

The processes $(\uV_t, \uW_t)_{t \in [\theta, T]}$ are coupled in that 
\begin{itemize}
  \item they are identical at time $\theta$: $\uW_\theta = \uV_\theta$;
  \item after time $\theta$, they are driven by the same PRMs $N^\tagi, \forall i \in [n]$;
  \item $\uW$ alone is not Markov and its jump rate depends on $\uV$. 
\end{itemize}
To ensure the well-posedness of the SDE solution, we verify that the jump rate is nonnegative and never explodes. The following lemma checks the range of $S_i$ and the next lemma verifies existence and uniqueness of the SDE solution. The proofs are deferred to Section~\ref{sub:proof_related_to_smoothness}. 

\begin{lemma}[Edge ratio bound]
  \label{lem:edge_ratio}
  For any nonnegative function $f: \braces{-1, 1}^n \to \real_+$ with $f \neq 0$, consider its multilinear expansion. Then for any $x\in (-1,1)^n$, $f$ is upper and lower bounded pointwise
  \begin{align*}
    0 < \vecnorm{f}{1} \prod_{j=1}^n (1-\abss{x_j}) \leq f(x) \leq \vecnorm{f}{1} \prod_{j=1}^n (1+\abss{x_j}).
  \end{align*}
  For any $i\in [n]$, the $i$-th edge ratio is bounded by
  \begin{align*}
    \frac{1-\abss{x_i}}{1+\abss{x_i}} \leq \frac{f(\flip_i(x))}{f(x)} \leq \frac{1+\abss{x_i}}{1-\abss{x_i}}, \quad \forall x \in (-1, 1)^n.
  \end{align*}
  Additionally, if $f$ is $\braces{0,1}$-valued, then its multilinear expansion takes values in $[0,1]$ on the inner cube $(-1,1)^n$.
\end{lemma}
\begin{lemma} [Uniqueness and existence]
  \label{lem:uniqueness_existence_SDE_sol}
  The coupled SDE for $(\uV_t, \uW_t)_{t \in [\theta, T]}$ in Eq.~\eqref{eq:def_reverse_process} and~\eqref{eq:def_SDE_perturbed_reverse_process} has a unique strong solution. Additionally, for $t \leq \To$ and $i \in [n]$, $0 \leq \delta_i < \btau$, where $\btau = \frac{1+e^{-\tau}}{1-e^{-\tau}}$. 
\end{lemma}

\paragraph{Predictable joint generator representation} The joint process $(\uV_t, \uW_t)_{t \in [\theta, T]}$ is not Markov because the rates depend on the predictable stopping indicator $\mathbf 1_{t\leq\stopT}$ and on the frozen
$\sF_\theta$-measurable perturbation size $\bdelta$. One could make
the process Markov by enlarging the state space to include this
frozen perturbation size and the stopping flag $\mathbf{1}_{t < \stopT}$. However, to simplify notation, we instead use the equivalent
predictable joint generator formulation.

Precisely, conditionally on $\sF_\theta$, $(\uV_t, \uW_t)_{t \in [\theta, T]}$ is a finite-state pure-jump process with predictable jump rates. By It\^o's formula, there exists a path-dependent and time-dependent operator $\cG_t$, predictable with respect to $(\sF_t)$, such that for every bounded $h: \braces{-1, 1}^{2n} \to \real$, 
\begin{align*}
  \cM_t^{h} &\defn h(\uV_t, \uW_t) - h(\uV_\theta, \uW_\theta) - \int_\theta^t \cG_s h(\uV_{s-}, \uW_{s-}) ds
\end{align*}
is a martingale on $[\theta, T]$. For simplicity, we refer to this operator $\cG_t$ as the predictable joint generator of $(\uV_t, \uW_t)_{t \in [\theta, T]}$. It takes the form
\begin{align}
  \label{eq:def_joint_generator}
  \cG_t \defn \jL^0_t + \mathbf{1}_{t \leq \stopT} A_t. 
\end{align}
Here, with shorthand $\rho_t = e^{-(T-t)}, \delta_i = \delta_i(\rho_t x), S_i = S_i(\rho_t x)$,
\begin{align}
  \label{eq:def_At}
  \jL_t^0 h(x, y) &= \sum_{i=1}^n  \parenth{\frac12 - S_i } \Delta_i^{xy} h(x, y), \notag \\
  A_t h (x, y) &= \sum_{i=1}^n  \mathbf{1}_{S_i > 0}  \delta_i  S_i \Delta_i^{y} h(x, y) + \mathbf{1}_{S_i \leq 0} \delta_i  S_i \Delta_i^{y} h(\flip_i(x), y)  . 
\end{align}
We make a few remarks about $\cG_t$. 
\begin{itemize}
  \item In words, $\cG_t$
  \begin{itemize}
    \item up to and including $\stopT$ (always $\leq \To$), coincides with the perturbed joint generator $\jL_t^0 + A_t$;
    \item after $\stopT$, coincides with the unperturbed joint generator $\jL_t^0$.
  \end{itemize}
  \item The path-dependence of $\cG_t$ enters through the frozen $\sF_\theta$-measurable perturbation size $\bdelta$ and through the predictable stopping indicator $\mathbf{1}_{t\leq\stopT}$. Conditioned on $\sF_\theta$, $\bdelta$ is fixed.
  \item Without perturbation, as captured in $\jL_t^0$, the two processes jump together. Thus, the coordinatewise product $x \odot y$ is preserved under the unperturbed dynamics.
\end{itemize}

\subsection{TV distance control}
\label{sub:proof_lem:coupling_TV_distance_control}

In this subsection, we upper-bound the TV distance between the laws of $\uV_{\To}$ and $\uW_{\To}$. Since $\tvdist{\Law(\uV_{\To})}{\Law(\uW_{\To})} = \sup_{\phi: \braces{-1,1}^n \to \braces{0, 1}} \abss{\Exs[ \phi(\uV_{\To} )] - \Exs[\phi(\uW_{\To})]}$, it suffices to consider any $\{0,1\}$-valued test function $\phi$. The main idea is to apply It\^o's formula to compare the perturbed and unperturbed joint processes from $\theta$ to $\To$ by treating the predictable joint generator~\eqref{eq:def_joint_generator} as unperturbed generator $\jL^0_t$ plus a small perturbation. The Doob $h$-transform of the process $(\uV_t)$ is crucial in the analysis because once we condition on $\uV_T = \zeta$, the coordinates of $\uV_t$ are independent and have a closed-form expression. The closed-form expression allows us to compute the perturbation effect explicitly.

Before completing the proof of Lemma~\ref{lem:coupling_TV_distance_control}, we need several lemmas, with proofs deferred to Section~\ref{sub:proof_related_to_tv_lemma}.

\begin{lemma}[Basic level-1 inequality]
  \label{lem:level_1_ineq}
  Let $h: \braces{-1,1}^n \to \braces{0, 1}$. Consider its multilinear expansion. For any $x \in (-1,1)^n$, we have
  \begin{align*}
    \sum_{i=1}^n (1-x_i^2) \parenth{\partial_i h(x)}^2 \leq h(x) - h(x)^2 \leq \frac14. 
  \end{align*}
\end{lemma}
This is a well-known consequence of Parseval's inequality for biased Fourier analysis, see, e.g., Section 3 of~\cite{eldan2023noise} or Appendix 1 of~\cite{eldan2020concentration}. 

Define
\begin{align}
  \label{eq:def_cS}
  \cS \defn \Exs \brackets{ \int_{\theta}^{\stopT} \sum_{i=1}^n \bdelta^2 S_i(\rV_{t-})^2 dt}
\end{align}
\begin{lemma}[Expected squared score bound] 
  \label{lem:expected_squared_score_bound}
  For the reverse heat process~\eqref{eq:def_reverse_process} and its scaled version $(\rV_t)$ in Eq.~\eqref{eq:def_scaled_reverse_jump_process}, we have
  \begin{align*}
    \Exs \brackets{ \int_{\theta}^{\stopT} \sum_{i=1}^n S_i(\rV_{t-})^2 dt \middle| \sF_\theta} \leq \btau (R_\theta + \excess + 1).
  \end{align*}
  Consequently, $\cS$ defined in~\eqref{eq:def_cS} satisfies
  \begin{align*}
    \cS \leq \Exs \brackets{\parenth{\frac{\excess \mathbf{1}_{R_\theta \geq \excess}}{R_\theta + 1} }^2 \btau (R_\theta + \excess + 1)} \leq 2 \btau \excess^2 \Exs\brackets{\frac{\mathbf{1}_{R_\theta \geq \excess}}{R_\theta + 1}}.
  \end{align*}
\end{lemma}
Intuitively, It\^o's formula applied to $\log f(\rV_t)$ relates the total squared score $\int_\theta^t \sum_{i=1}^n S_i(\rV_{s-})^2 ds$ to $\log f(\rV_t) - \log f(\rV_\theta)$, up to a multiplicative factor.

\begin{lemma}[Boolean heat bridge]
  \label{lem:boolean_bridge_formula}
  For $0 \leq t \leq \To$ and $x, y, \zeta \in \braces{-1,1}^n$, define
\begin{align}
  \label{eq:def_m}
  m_{t}(x, y, \zeta) \defn \Exs\brackets{\uV_{\To} \odot (x \odot y) \middle| \uV_t = x, \uV_{T} = \zeta}.
\end{align}
The coordinates of $\uV_{\To}$ are conditionally independent given $(\uV_t, \uV_{T})$ and their conditional law has a closed form. For $i \in [n]$, the $i$-th component of $m_{t}(x, y, \zeta)$ is given by  
  \begin{align*}
     m_{t}^{[i]}(x, y, \zeta) = a_{t} y_i + b_{t} \sigma_i
  \end{align*}
  where $\sigma_i \defn x_i y_i \zeta_i$ and
  \begin{align}
    \label{eq:def_a_b}
    a_{t} &\defn \frac{\sinh(T-\To)}{\sinh(T-t)}, \quad b_{t} \defn \frac{\sinh(\To-t)}{\sinh(T-t)}.
  \end{align}
  Because the coordinates of $\uV_{\To}$ are conditionally independent, for any multilinear $\phi:\braces{-1,1}^n \to \real$,
  \begin{align*}
    \Exs\brackets{\phi(\uV_\To \odot (x \odot y))  \middle| \uV_t = x, \uV_T = \zeta} = \phi(m_{t}(x, y, \zeta)).
  \end{align*}
  We define $q_t^{\zeta}(x, y) \defn \phi(m_{t}(x, y, \zeta))$. Then
  \begin{align}
    \label{eq:q_t_derivative}
    \Delta_i^y q_t^{\zeta}(x, y) &= -2 (a_t y_i + b_t \sigma_i) \partial_i \phi(m_{t}(x, y, \zeta)), \notag \\
    \Delta_i^y q_t^{\zeta}(\flip_i(x), y) &= -2 (a_t y_i - b_t \sigma_i) \partial_i \phi(m_{t}(x, y, \zeta)), \notag \\
    \Delta_{i}^{xy} q_t^{\zeta}(x, y) &= -2 a_t y_i \partial_i \phi(m_{t}(x, y, \zeta)).
  \end{align}
\end{lemma}

\begin{lemma}[Doob $h$-transform of the predictable joint generator]
  \label{lem:Doob_h_transform}
  Given $\zeta \in \braces{-1,1}^n$, define
  \begin{align}
    \label{eq:def_H_r_lambda}
    H_t^\zeta(x) &\defn \Prob(\uV_T = \zeta \mid \uV_t = x), \notag \\
    r_{t, i}^{\zeta}(x) &\defn \frac{H_t^\zeta(\flip_i(x))}{H_t^\zeta(x)}, \notag \\
    \lambda_{t, i}^{\zeta}(x) &\defn \frac{1 - \rho_t x_i \zeta_i}{1 + \rho_t x_i \zeta_i}, \quad \forall x \in \braces{-1, 1}^n.
  \end{align}
  Under the conditioned law $\Prob^\zeta \defn \Prob(\cdot \mid \uV_T = \zeta)$,  with shorthand $\rho_t = e^{-(T-t)}$, $\delta_i = \delta_i(\rho_t x)$, $S_i = S_i(\rho_t x)$, $r_i^{\zeta} = r_{t, i}^{\zeta}(x)$, $\lambda_i^{\zeta} = \lambda_{t, i}^{\zeta}(x)$, for every bounded function $h: \braces{-1, 1}^{2n} \to \real$, 
  \begin{align*}
    \cM_t^{\zeta, h} &\defn h(\uV_t, \uW_t) - h(\uV_\theta, \uW_\theta) - \int_\theta^t \cG_s^\zeta h(\uV_{s-}, \uW_{s-}) ds
  \end{align*}
  is a $\Prob^\zeta$-martingale on $[\theta, T]$, where $\cG_t^\zeta$ is the Doob $h$-transform of the predictable joint generator $\cG_t$ with respect to the function $H_t^\zeta$, and it takes the form
  \begin{align*}
    \cG_t^\zeta &= \jL_t^{0, \zeta} + \mathbf{1}_{t \leq \stopT} A_t^\zeta, \\
    \jL_t^{0, \zeta} h(x, y) &= \sum_{i=1}^n r_{i}^{\zeta} \parenth{\frac12 - S_i } \Delta_i^{xy} h(x, y) \\
    A_t^\zeta h(x, y) &= \sum_{i=1}^n  \mathbf{1}_{S_i > 0}  \delta_i  S_i \Delta_{i}^y h(x, y)  + \mathbf{1}_{S_i \leq 0} r_{i}^{\zeta} \delta_i  S_i  \Delta_i^y h(\flip_i(x), y).
  \end{align*}
  Moreover,
  \begin{align*}
    r_i^{\zeta} \parenth{\frac12 - S_i } &=  \frac12 \lambda_{i}^\zeta.
  \end{align*}
\end{lemma}
\begin{remark}
  The key reason that the Doob $h$-transform works for the predictable joint generator as it does for general Markov semigroups is that the conditioning weight $H_t^{\zeta}$ depends only on $\uV_t$, even though the joint process is path-dependent through $\mathbf{1}_{t \leq \stopT}$. Both the Boolean heat bridge in Lemma~\ref{lem:boolean_bridge_formula} and the Doob $h$-transform in Lemma~\ref{lem:Doob_h_transform} become effective proof techniques for the TV distance control in Lemma~\ref{lem:coupling_TV_distance_control} under the heat semigroup $P_\tau$ with $\tau >0$. That is, for $t < \To$, we have
  \begin{align*}
    0 < a_t + b_t = \frac{\cosh\parenth{\frac{\tau- (\To - t)}{2}}}{\cosh\parenth{\frac{\tau+(\To - t)}{2}}} < 1.
  \end{align*}
  As long as $\tau > 0$, $m_t$ lies strictly inside the inner cube $(-1,1)^n$ and also the conditional unperturbed jump rate $\frac12 \lambda_{t, i}^{\zeta}(x)$ is strictly positive and finite. 
\end{remark}

\begin{lemma}[Weighted energy estimate]
  \label{lem:weighted_energy_estimate}
  Given a function $\phi:\braces{-1,1}^n \to \braces{0, 1}$. For $a_{t}$, $b_{t}$ and $m_{t}$ defined in Lemma~\ref{lem:boolean_bridge_formula}, $\lambda_{t, i}^{\zeta}(x)$ defined in Lemma~\ref{lem:Doob_h_transform}, define
  \begin{align*}
    \Psi_a &\defn \Exs \int_\theta^{\To} a_{t}^2 \sum_{i=1}^n \lambda_{t, i}^{\uV_T}(\uV_t) \abss{\partial_i \phi(m_{t}(\uV_t, \uW_t, \uV_T))}^2 dt, \\
    \Psi_b &\defn \Exs \int_\theta^{\To} b_{t}^2 \sum_{i=1}^n \lambda_{t, i}^{\uV_T}(\uV_t) \abss{\partial_i \phi(m_{t}(\uV_t, \uW_t, \uV_T))}^2 dt.
  \end{align*}
  Then, with $\cS$ defined in Lemma~\ref{lem:expected_squared_score_bound}, we have
  \begin{align*}
    \Psi_b &\leq \frac{e^{-2\tau}}{1-e^{-2\tau}} \parenth{\To - \theta}, \\
    \Psi_a &\lesssim 1 + \btau \cS + \Psi_b.
  \end{align*}
\end{lemma}

With the lemmas above, we are ready to prove Lemma~\ref{lem:coupling_TV_distance_control}.
\begin{proof}[Proof of Lemma~\ref{lem:coupling_TV_distance_control}]
  Fix a test function $\phi:\braces{-1,1}^n \to \braces{0, 1}$ and set a bivariate test function such that $h(x, y) = \phi(y)$. We bound $\abss{\Exs \phi(\uW_{\To}) - \Exs \phi(\uV_{\To})}$ and then take the supremum over $\phi$ to prove a total-variation bound.

  First, we use It\^o's formula to provide a pathwise representation. Let $(\uV_t, \uW_{t}^{(0)})_{t \in [\theta, T]}$ denote the unperturbed joint process with generator $\jL_t^0$, define 
  \begin{align*}
    u_t(x, y) \defn \Exs[h(\uV_{\To}, \uW_{\To}^{(0)}) \mid (\uV_t, \uW_t^{(0)}) = (x, y)], \quad \text{ for } (t, x, y) \in [\theta, \To] \times \braces{-1,1}^{2n}.
  \end{align*}
  By definition, $u_t$ solves the backward equation $\partial_t u_t + \jL_t^0 u_t = 0$ with the terminal condition $u_{\To} = h$. Applying It\^o's formula to $u_t(\uV_t, \uW_t)$, we obtain
  \begin{align*}
    u_{\To}(\uV_{\To}, \uW_{\To}) - u_\theta(\uV_\theta, \uW_\theta) &= \int_\theta^{\To} \parenth{\partial_t u_t + \cG_t u_t}(\uV_{t-}, \uW_{t-}) dt + \cM_{\To},
  \end{align*}
  where $(\cM_{t})_{t \in [\theta, \To]}$ is a martingale which is $0$ at $\theta$.
  Since $\partial_t u_t + \jL_t^0 u_t = 0$, we have $\partial_t u_t + \cG_t u_t = \mathbf{1}_{t \leq \stopT} A_t u_t$. Taking expectation and using the fact that $\uV_\theta = \uW_\theta$, we obtain
  \begin{align}
    \label{eq:pathwise_formula}
    \Exs \phi(\uW_{\To}) - \Exs \phi(\uV_{\To})
    &= \Exs u_{\To}(\uV_{\To}, \uW_{\To}) - \Exs u_{\To}(\uV_{\To}, \uW_{\To}^{(0)}) \notag \\
    &= \Exs u_{\To}(\uV_{\To}, \uW_{\To}) - \Exs u_\theta(\uV_\theta, \uW_\theta) \notag \\
    &= \Exs \int_\theta^{\To} \mathbf{1}_{t \leq \stopT} A_t u_t(\uV_{t-}, \uW_{t-}) dt,
  \end{align}

  Second, we rewrite $u_t$ in terms of the Boolean heat bridge. We have
  \begin{align*}
    u_t(x, y) &= \Exs[\phi(\uW_{\To}^{(0)}) \mid (\uV_t, \uW_t^{(0)}) = (x, y)] \\
    &\overset{(i)}{=} \Exs[\phi(\uV_{\To} \odot (x \odot y)) \mid (\uV_t, \uW_t^{(0)}) = (x, y)] \\
    &\overset{(ii)}{=} \Exs\brackets{ \Exs[\phi(\uV_{\To} \odot (x \odot y)) \mid (\uV_t, \uW_t^{(0)}, \uV_{T}) = (x, y, \uV_{T}) ] } \\
    &\overset{(iii)}{=} \Exs\brackets{ \Exs[\phi(\uV_{\To} \odot (x \odot y)) \mid (\uV_t, \uV_{T}) = (x, \uV_{T}) ] } \\
    &\overset{(iv)}{=} \Exs[\phi(m_{t}(x, y, \uV_{T})) \mid \uV_t = x] \\
    &\overset{(v)}{=} \Exs[q_t^{\uV_{T}}(x, y) \mid \uV_t = x],
  \end{align*}
  where (i) uses the fact that $(\uV_t, \uW_t^{(0)})$ have synchronized jumps; (ii) uses the tower property; (iii) follows from the fact that $(\uV_t, \uV_T)$ is sufficient for $\uV_{\To}$ from Lemma~\ref{lem:boolean_bridge_formula}; (iv) and (v) uses the definition of $m_t$ and $q_t$ in Lemma~\ref{lem:boolean_bridge_formula}. 

  Using the definition of $A_t$ and the representation for $u_t$, we obtain
  \begin{align*}
     &\quad A_t u_t(x, y) \\
     &= \sum_{i=1}^n  \mathbf{1}_{S_i > 0}  \delta_i  S_i \Exs\brackets{\Delta_i^y q_t^{\uV_T}(x, y) \mid \uV_t = x} + \mathbf{1}_{S_i \leq 0} \delta_i  S_i \Exs\brackets{\Delta_i^y q_t^{\uV_T}(\flip_i(x), y) \mid \uV_t = \flip_i(x)} \\
     &\overset{(i)}{=} \sum_{i=1}^n  \mathbf{1}_{S_i > 0}  \delta_i  S_i \sum_{\zeta} H_{t}^{\zeta}(x) \brackets{\Delta_i^y q_t^{\zeta}(x, y)} + \mathbf{1}_{S_i \leq 0} \delta_i  S_i \sum_{\zeta} H_t^{\zeta}(\flip_i(x))\brackets{\Delta_i^y q_t^{\zeta}(\flip_i(x), y)} \\
     &= \sum_{\zeta} H_{t}^{\zeta}(x) \sum_{i=1}^n \delta_i S_i \brackets{\mathbf{1}_{S_i > 0} + r_i^{\zeta} \mathbf{1}_{S_i \leq 0} } \brackets{\mathbf{1}_{S_i > 0} \Delta_i^y q_t^{\zeta}(x, y) + \mathbf{1}_{S_i \leq 0} \Delta_i^y q_t^{\zeta}(\flip_i(x), y)}.
  \end{align*}
  where $(i)$ expands the conditional expectation with respect to $\uV_T$ and uses the definition of $H_t^{\zeta}$ in Lemma~\ref{lem:Doob_h_transform}, and the last step uses $r_i^{\zeta} = r_{t, i}^{\zeta}(x) = \frac{H_t^{\zeta}(\flip_i(x))}{H_t^{\zeta}(x)}$. 

  Note that from Lemma~\ref{lem:Doob_h_transform} and $\btau^{-1} \leq \lambda_{t, i}^{\zeta}(x) \leq \btau$,
  \begin{align*}
    \abss{\delta_i \brackets{\mathbf{1}_{S_i > 0} + r_i^{\zeta} \mathbf{1}_{S_i \leq 0} }}^2 = \abss{\bdelta \brackets{\mathbf{1}_{S_i > 0} + \frac{\lambda_{t, i}^{\zeta}(x)}{1 - 2 \bdelta S_i} \mathbf{1}_{S_i \leq 0}}}^2
    \leq \bdelta^2 \btau  \lambda_{t, i}^{\zeta}(x).
  \end{align*}
  Therefore,
  \begin{align}
    \label{eq:bound_At_u_t}
    \abss{A_t u_t(x, y)} &\leq \sum_{\zeta} H_t^{\zeta}(x) \sum_{i=1}^n \abss{\delta_i S_i \brackets{\mathbf{1}_{S_i > 0} + r_i^{\zeta} \mathbf{1}_{S_i \leq 0} }} \sqrt{\abss{\Delta_i^y q_t^{\zeta}(x, y)}^2 + \abss{\Delta_i^y q_t^{\zeta}(\flip_i(x), y)}^2} \notag \\
    &\leq \btau^{\frac12} \sum_{\zeta} H_t^{\zeta}(x)  \parenth{\sum_{i=1}^n \bdelta^2 S_i^2}^{\frac12} \brackets{\sum_{i=1}^n \lambda_{t, i}^{\zeta}(x) \parenth{\abss{\Delta_i^y q_t^{\zeta}(x, y)}^2 + \abss{\Delta_i^y q_t^{\zeta}(\flip_i(x), y)}^2}  }^{\frac12} \notag \\
    &\leq \btau^{\frac12} \parenth{\sum_{i=1}^n \bdelta^2 S_i^2}^{\frac12} \brackets{ \sum_{\zeta} H_t^{\zeta}(x)  \sum_{i=1}^n \lambda_{t, i}^{\zeta}(x) \parenth{\abss{\Delta_i^y q_t^{\zeta}(x, y)}^2 + \abss{\Delta_i^y q_t^{\zeta}(\flip_i(x), y)}^2}  }^{\frac12},
  \end{align}
  where the last step uses Jensen's inequality and the concavity of square root.
  Using the identities from Lemma~\ref{lem:boolean_bridge_formula}, we have
  \begin{align*}
    \abss{\Delta_i^y q_t^{\zeta}(x, y)}^2 + \abss{\Delta_i^y q_t^{\zeta}(\flip_i(x), y)}^2 = 8 (a_t^2 + b_t^2) \abss{\partial_i \phi(m_{t}(x, y, \zeta))}^2.
  \end{align*}
  Plugging Eq.~\eqref{eq:bound_At_u_t} with the above identity into Eq.~\eqref{eq:pathwise_formula} and applying Cauchy-Schwarz in time and in expectation, we obtain
  \begin{align*}
    \abss{\Exs \phi(\uW_{\To}) - \Exs \phi(\uV_{\To})} &\lesssim \btau^{\frac12} \parenth{\Exs \brackets{\int_\theta^{\stopT} \bdelta^2 \sum_{i=1}^n S_i(\rV_{t-})^2 dt}}^{\frac12} \parenth{ \Psi_a + \Psi_b}^{\frac12} \\
    & = \btau^{\frac12} \cS^{\frac12} \parenth{ \Psi_a + \Psi_b}^{\frac12}.
  \end{align*}
  Applying the bounds for $\Psi_a$ and $\Psi_b$ in Lemma~\ref{lem:weighted_energy_estimate} and the bound for $\cS$ in Lemma~\ref{lem:expected_squared_score_bound}, we have
  \begin{align*}
    \abss{\Exs \phi(\uW_{\To}) - \Exs \phi(\uV_{\To})} &\lesssim \btau^{\frac12} \cS^{\frac12} \parenth{1 + \btau \cS + \Psi_b}^{\frac12} \\
    &\lesssim \btau^{\frac12} \cS^{\frac12} \parenth{1 + \btau \cS + \frac{e^{-2\tau}}{1-e^{-2\tau}} (\To - \theta) }^{\frac12} \\
    &\lesssim \btau^2 \excess^2 I_\theta + \sqrt{\btau^2 \excess^2 I_\theta}\sqrt{1 + \frac{e^{-2\tau}}{1-e^{-2\tau}} (\To - \theta)}.
  \end{align*}
  Conclude by taking the supremum over $\phi$.

\end{proof}

\begin{remark}
  The $b_t$ term in the bound for $\Psi_b$ in Lemma~\ref{lem:weighted_energy_estimate} is the only place that causes dependence on the maximal perturbation duration $\parenth{\To - \theta}$. This is an obstruction specific to the Boolean setting because the same proof strategy for the Gaussian counterpart does not produce such a term.
\end{remark}
\begin{remark}\label{rmk:remark_after_lem2}
  A natural alternative idea to bound the TV distance is through a Kullback-Leibler (KL) divergence bound and Pinsker's inequality. The KL divergence can be bounded via Girsanov's theorem applied to jump processes (see e.g.~\cite{leonard2012girsanov}). However, to bound KL divergence this way requires a bound on $\sum_{i=1}^n \parenth{S_i(\rho_t \uV_t) - S_i(\rho_t \uW_t)}^2$, and we do not know how to obtain a good bound of this type. Intuitively, while in the Gaussian counterpart the perturbed process and the original process are close in $L^2$ distance, the discrete nature of the Boolean hypercube $\braces{-1,1}^n$ makes them far apart in $L^2$ distance. This seems to be the main difficulty in working with the perturbed reverse process on the Boolean hypercube. 
  
  Another idea to bound the TV distance is to follow the proof in~\cite{eldan2018regularization} and~\cite{lehec2016regularization}. They crucially use the semi-log-convexity gained from the heat semigroup (see Eq. (5.6) in~\cite{lehec2016regularization}), and the second-order Taylor expansion derived from it. As we show in Lemma~\ref{lem:hessian_smoothness} in Appendix, the Boolean heat semigroup has an analogous semi-log-convexity, giving
  \begin{align*}
    \log P_\tau f(\uW_{\To}) \geq \log P_\tau f(\uV_{\To}) + \angles{\nabla \log P_\tau f (\uV_{\To}), \uW_{\To} -\uV_{\To}} - \frac{c_\tau}{2} \vecnorm{\uW_{\To} - \uV_{\To}}{2}^2,
  \end{align*}
  for a constant $c_\tau$ that depends only on $\tau$. However, we do not know how to obtain a good dimension-free $L^2$ bound on $\vecnorm{\uW_{\To} - \uV_{\To}}{2}$, which is intuitively quite large as the distance between two points on the Boolean hypercube. The present proof is designed to avoid such $L^2$ bounds.
\end{remark}

\subsection{Proof of Lemma~\ref{lem:approx_monotone_coupling} on the approximate monotonicity at tail}
\label{sub:proof_lem:approx_monotone_coupling}
In this section, we show that given an anti-concentration bound at the starting time $\theta$, our coupling construction $(\uV, \uW)$ makes $\log P_{\tau}f (\uV_{\To})$ strictly larger than $\log P_{\tau}f (\uW_{\To})$ conditioned on $\log P_{\tau}f (\uW_{\To})$ being large with high probability, via stochastic calculus. The perturbation in Eq.~\eqref{eq:def_delta_i} is designed such that $\log P_{\tau}f (\uV_{\To}) - \log P_{\tau}f (\uW_{\To})$ can be lower bounded by a constant times $\log P_{T-\stopT}f (\uW_{\stopT}) - \log P_{T-\theta}f (\uW_{\theta})$ up to small terms. In the following, it is more convenient to work with scaled processes, with $\rho_t = e^{-(T-t)}$,
\begin{align*}
  \rV_t &= \rho_t \uV_t, \\
  \rW_t &\defn \rho_t \uW_t.
\end{align*}
Then $P_{T-t}f (\uV_t) = f(\rV_t)$ and $P_{T-t}f (\uW_t) = f(\rW_t)$.
Applying It\^o's formula to $\log f(\rV_t)$ and $\log f(\rW_t)$, we obtain
\begin{align*}
  d \log f(\rV_t) &= \sum_{i=1}^n \brackets{S_i(\rV_{t-}) dt + \int \log \parenth{1-2S_i(\rV_{t-})} \mathbf{1}_{0 < z \leq \frac12 - S_i(\rV_{t-})} \Ni(dt, dz) }, \\
  d \log f(\rW_t) &= \sum_{i=1}^n \brackets{S_i(\rW_{t-}) dt + \int \log \parenth{1-2S_i(\rW_{t-})} \mathbf{1}_{0 < z \leq \frac12 - (1 - \delta_i \mathbf{1}_{t \leq \stopT})S_i(\rV_{t-})} \Ni(dt, dz) },
\end{align*}
where recall that $\delta_i \equiv \delta_i (\rV_{t-})$ and
\begin{align*}
  \delta_i(x) = \bdelta  \brackets{\mathbf{1}_{S_i(x) > 0} + \frac{1-2 S_i(x)}{1 - 2 \bdelta S_i(x)} \mathbf{1}_{S_i(x) \leq 0}}, \quad i \in [n],\quad x\in [-1,1]^n.
\end{align*}
Introduce the shorthand $(v_i, w_i) \defn (S_i(\rV_{t-}), S_i(\rW_{t-}))$, where we omit the dependence on $t$ as long as it is clear from the context. Then the two SDEs above become
\begin{align}
  \label{eq:logf_VW_simple_notation}
  d \log f(\rV_t) &= \sum_{i=1}^n \brackets{v_i dt + \int \log \parenth{1-2v_i} \mathbf{1}_{0 < z \leq \frac12 - v_i} \Ni(dt, dz) }, \notag \\
  d \log f(\rW_t) &= \sum_{i=1}^n \brackets{w_i dt + \int \log \parenth{1-2w_i} \mathbf{1}_{0 < z \leq \frac12 - (1 - \delta_i \mathbf{1}_{t \leq \stopT})v_i} \Ni(dt, dz)}.
\end{align}
To complete the proof of Lemma~\ref{lem:approx_monotone_coupling}, we need three lemmas to compare the two SDEs in Eq.~\eqref{eq:logf_VW_simple_notation} with high probability. Their proofs are deferred to Section~\ref{sub:proof_related_to_monotone_coupling}.
\begin{lemma}
  \label{lem:logfV_lower_bound}
  With probability at least $1 - \frac{1}{\sqrt{\log \eta}}$, we have
  \begin{align*}
    \log f (\rV_{\To}) >  \log f (\rV_{\stopT}) - \frac{1}{2} \log\log \eta.
  \end{align*}
\end{lemma}
Since $\stopT$ is always upper bounded by $\To$, from Eq.~\eqref{eq:logf_VW_simple_notation}, $\log f (\rV_{\To}) -  \log f (\rV_{\stopT})$ has nonnegative drift terms after compensating the Poisson jumps. Lemma~\ref{lem:logfV_lower_bound} states that the compensated Poisson noise does not alter the nonnegativity too much.  
\begin{lemma}
  \label{lem:wivi_bound}
  With probability at least $1 - \frac{1}{\sqrt{\log \eta}}$, we have
  \begin{align*}
    \log f (\rW_{\stopT}) -  \log f (\rW_{\theta}) < \brackets{2 \sum_{i=1}^n \int_\theta^\stopT   (1-\delta_i) w_i v_i dt } + \frac{1}{2} \log\log \eta.
  \end{align*}
\end{lemma}
In words, Lemma~\ref{lem:wivi_bound} upper bounds $\log f (\rW_{\stopT})$ in terms of the corresponding score terms with high probability.  
\begin{lemma}
  \label{lem:log_diff_bound}
  With probability at least $1-\frac{1}{\sqrt{\log \eta}}$, we have 
  \begin{align*}
    \log f(\rW_\To) - \log f(\rV_\To) < \brackets{- 2 \sum_{i=1}^n \int_\theta^{\stopT} \delta_i \parenth{\mathbf{1}_{v_i > 0} + \mathbf{1}_{v_i \leq 0} \frac{1}{1-2v_i} } w_i v_i  dt } + \frac12 \log\log \eta. 
  \end{align*}
\end{lemma}
Roughly speaking, Lemma~\ref{lem:log_diff_bound} shows that $\log f(\rV_\To)$ is larger than $\log f(\rW_\To)$ by a term proportional to $\bdelta$ with high probability. We are now ready to prove Lemma~\ref{lem:approx_monotone_coupling}. 
\begin{proof}[Proof of Lemma~\ref{lem:approx_monotone_coupling}]
  Recall that $\excess = \frac12 \log \log \eta + 1$. Define the events
  \begin{align*}
    \cE_0 &\defn \braces{\log f(\rW_\theta) \notin (\log \eta - \excess, \log \eta + \excess]}, \\
    \cE_1 &\defn \braces{\log f (\rV_{\To}) > \log f (\rV_{\stopT}) - \excess + 1}, \\
    \cE_2 &\defn \braces{\log f (\rW_{\stopT}) -  \log f (\rW_{\theta}) < \brackets{2 \sum_{i=1}^n \int_\theta^\stopT   (1-\delta_i) w_i v_i dt } + (\excess - 1)}, \\
    \cE_3 &\defn \braces{\log f(\rW_\To) - \log f(\rV_\To) < \brackets{- 2 \sum_{i=1}^n \int_\theta^{\stopT} \delta_i \parenth{\mathbf{1}_{v_i > 0} + \mathbf{1}_{v_i \leq 0} \frac{1}{1-2v_i} } w_i v_i dt } + (\excess - 1)}. 
  \end{align*}
  Since $\uW_\theta = \uV_\theta \sim \nu_{P_{T-\theta} f}$ and $f(\rW_\theta) = P_{T-\theta} f(\uW_\theta)$, by the definition of $\fA$ in Eq.~\eqref{eq:def_anti_concentration_profile}, we have
  \begin{align*}
    \Prob(\cE_0^c) = \fA_{T - \theta}((\log \eta - \excess, \log \eta + \excess]).
  \end{align*}
  Together with Lemma~\ref{lem:logfV_lower_bound}, \ref{lem:wivi_bound} and \ref{lem:log_diff_bound}, we have
  \begin{align}
    \label{eq:E123_bound}
    \Prob(\cE_0 \cap \cE_1 \cap \cE_2 \cap \cE_3) \geq 1 - \frac{3}{\sqrt{\log \eta}} - \fA_{T - \theta}((\log \eta - \excess, \log \eta + \excess]).
  \end{align}
  The rest of the proof is conditioned on the intersection of all four events. Our choice of $\delta_i$ allows us to connect the two quantities on the right-hand side. Recall that
\begin{align*}
  \delta_i = \bdelta  \brackets{ \mathbf{1}_{v_i > 0} + \frac{1-2v_i}{1 - 2 \bdelta v_i} \mathbf{1}_{v_i \leq 0}} , \quad i \in [n],\quad x\in [-1,1]^n.
\end{align*}
Observe that
\begin{align*}
  1 - \delta_i &= (1-\bdelta) \brackets{\mathbf{1}_{v_i > 0} + \frac{1}{1-2\bdelta v_i} \mathbf{1}_{v_i \leq 0}} \\
  \delta_i \parenth{\mathbf{1}_{v_i > 0} + \mathbf{1}_{v_i \leq 0} \frac{1}{1-2v_i}}
  &= \bdelta\brackets{\mathbf{1}_{v_i > 0} + \frac{1}{1-2 \bdelta v_i} \mathbf{1}_{v_i \leq 0} }
  \overset{(i)}{=} \frac{\bdelta}{1-\bdelta} (1-\delta_i). 
\end{align*}
(i) uses the expression of $1-\delta_i$ in the first line. Consequently, 
\begin{align*}
  \delta_i w_i v_i \parenth{\mathbf{1}_{v_i > 0} + \mathbf{1}_{v_i \leq 0} \frac{1}{1-2v_i} } = \frac{\bdelta}{1-\bdelta}  \brackets{(1-\delta_i) w_i v_i}.
\end{align*}
This identity combines the estimates on $\cE_2$ and $\cE_3$. Specifically, under $\cE_2 \cap \cE_3$,
\begin{align*}
  \log f(\rW_\To) - \log f(\rV_\To) < -\frac{\bdelta}{1-\bdelta} \brackets{\log f(\rW_\stopT) - \log f(\rW_\theta) - (\excess - 1)} + (\excess - 1).
\end{align*}
Rearranging, we obtain
\begin{align}
  \label{eq:logf_lower_bound_from_stochastic_calculus}
  \log f(\rV_\To) >  \log f(\rW_\To) + \frac{\bdelta}{1-\bdelta} \brackets{\log f(\rW_\stopT) - \log f(\rW_\theta)} - \frac{1}{(1-\bdelta)}(\excess - 1). 
\end{align}
The event $\braces{\log f(\rW_\To) > \log \eta}$ implies that, based on our definition of stopping time~\eqref{eq:def_stopT}, 
\begin{align*}
  \text{ either } \log f(\rV_\stopT) \geq \log \eta + \excess \text{ or } \log f(\rW_\stopT) \geq \log\eta.
\end{align*}
If the first case holds at $\stopT$, then by Lemma~\ref{lem:logfV_lower_bound}, using $\cE_1$
\begin{align*}
  \log f(\rV_\To) > \log f(\rV_\stopT) - \excess + 1  \geq \log \eta + 1.
\end{align*}
If the second case holds at $\stopT$, we split $\cE_0$ into two cases. If $\log f(\rW_\theta) = \log f(\rV_\theta) > \log \eta + \excess$, then $\log f(\rV_\stopT) > \log \eta + \excess$ and the preceding case applies. Otherwise $\log f(\rW_\theta) = \log f(\rV_\theta) \leq \log \eta - \excess$, which implies $R_\theta \geq \excess$. Using Eq.~\eqref{eq:logf_lower_bound_from_stochastic_calculus} and $\cE_2 \cap \cE_3$, we obtain
\begin{align*}
  \log f(\rV_\To) &>  \log \eta + \frac{1}{1-\bdelta}\brackets{\bdelta  R_\theta - (\excess - 1)} \\
  &\overset{(i)}{\geq} \log \eta + \frac{1}{1-\bdelta} \brackets{1 - \bdelta} \\
  &= \log \eta + 1.
\end{align*}
where (i) follows from $\bdelta (R_\theta + 1) \geq \excess$ when $R_\theta \geq \excess$. Combining both cases, we obtain the following implication 
\begin{align*}
  \cE_0 \cap \cE_1 \cap \cE_2 \cap \cE_3 \cap \braces{\log f(\rW_\To) > \log \eta} \implies
  \braces{\log f(\rV_\To) > \log \eta + 1}. 
\end{align*}
Together with Eq.~\eqref{eq:E123_bound}, we conclude that
\begin{align*}
  \Prob\parenth{\log f(\rV_\To) > \log \eta + 1}
    \geq \Prob\parenth{\log f(\rW_\To) > \log \eta} - \frac{3}{\sqrt{\log \eta}} - \fA_{T-\theta}((\log \eta - \excess, \log \eta + \excess]).
\end{align*}
\end{proof}

\begin{remark}
  The main difference between Poisson noise in the Boolean setting and Brownian motion in the Gaussian setting is that the former is not symmetric. This asymmetry is why we have to design state-dependent perturbation and why we lower bound the unperturbed process by the perturbed process while the inequality is reversed in the Gaussian counterpart~\cite{eldan2018regularization}. Lower bounding the noise in Lemma~\ref{lem:log_diff_bound} is more challenging than upper bounding it. 

  In view of the exponential martingale bound in Lemma~\ref{lem:log_diff_bound}, it is reasonable to conjecture that one might be able to obtain a better bound to eliminate the extra $\log\log\eta$ factor via a similar strategy. One possibility is to obtain a better $L^2$ bound on $\sum_{i} (w_i - v_i)^2$ and use a larger $\gamma$ in Eq.~\eqref{eq:exponentiated_process_lem11} in the proof of Lemma~\ref{lem:log_diff_bound}. However, similar to the discussion in Remark~\ref{rmk:remark_after_lem2}, we do not know how to prove a good dimension-free $L^2$ bound. The present proof is designed to avoid such terms.
\end{remark}

\subsection{Proof of Lemma~\ref{lem:anticoncentration_profile_bound} on time-smoothed anti-concentration profile}
\label{sub:time_smoothed_anti_concentration_profile_bound}
In this section, we show that the time-smoothed anti-concentration profile is upper-bounded at a rate $1/\ell$. 
The main idea is to relate the anti-concentration profile $\fA_{s}((\ell, \ell+1])$ to the entropy of a localized test function, and then apply the log-Sobolev inequality.

Recall from Eq.~\eqref{eq:heat_semigroup_generator} that the generator of the heat semigroup on the Boolean hypercube takes the form 
  $\fL h = \frac{1}{2} \sum_{i=1}^n \Delta_i h$. Define its associated Dirichlet form as 
  \begin{align}
    \label{eq:def_Dirichlet_form}
    \dirichlet(g, h) \defn - \int g \fL h d\mu = \frac14 \int \sum_{i=1}^n \Delta_i g \Delta_i h d\mu, \quad g, h: \braces{-1, 1}^n \to \real.
  \end{align}

\begin{proof}[Proof of Lemma~\ref{lem:anticoncentration_profile_bound}]
  Let $E \defn \braces{x \in \braces{-1,1}^n \mid \log P_s f(x) \in (\ell, \ell + 1]}$. Define a localized test function $h_s: \braces{-1, 1}^n \to [0, \infty)$
  \begin{align*}
    h_s(x) \defn P_s f(x) \chi(\log P_s f(x) - \ell)^2,
  \end{align*}
  where $\chi: \real \to [0,1]$ is a $C^1$ cutoff function such that $\chi = 1$ on $[0, 1]$, $\supp \chi \subseteq (-1, 2)$ and $\sup_{x \in \real} \abss{\chi'(x)} \leq c'$ for a universal constant $c'$. Then $h_s = P_s f$ on $E$, and $h_s = 0$ when $\log P_s f(x) \notin (\ell - 1, \ell + 2)$. $h_s$ is designed to localize the mass of $P_s f$ on the set $E$ and to be smooth enough.

  % We first relate $\fA_{s}((\ell, \ell+1])$ to $\Ent_\mu h_s$. Then we apply the log-Sobolev inequality to upper bound $\Ent_\mu h_s$ by the Dirichlet form $\dirichlet(\sqrt{h_s}, \sqrt{h_s})$. Finally, we prove a discrete coarea formula to upper bound $\dirichlet(\sqrt{h_s}, \sqrt{h_s})$ by an integral of $\dirichlet(\mathbf{1}_{P_s f > e^{u}}, P_s f)$, which has a bounded integral over $s$.
  
  First, we relate $\fA_{s}((\ell, \ell+1])$ to the entropy $\Ent_\mu h_s$. Applying the variational form of entropy for $h: \braces{-1, 1}^n \to [0, \infty)$, we have
  \begin{align*}
    \Ent_\mu(h) = \sup \braces{\int h \varphi d\mu \ \middle| \ \int e^{\varphi} d\mu \leq 1}.
  \end{align*}
  Take $\varphi = \ell \mathbf{1}_{E} - \log \brackets{1 + (e^{\ell} - 1) \mu(E)}$. We verify that
  \begin{align*}
    \int e^{\varphi} d\mu = \frac{1}{1+ (e^{\ell} - 1) \mu(E)} \int e^{\ell \mathbf{1}_{E}} d\mu = \frac{e^{\ell}}{1 + (e^{\ell} - 1) \mu(E)} \mu(E) + \frac{1}{1 + (e^{\ell} - 1) \mu(E)} (1 - \mu(E)) = 1.
  \end{align*}
  Then
  \begin{align}
    \label{eq:entropy_vs_anti_concentration_profile}
    \Ent_\mu(h_s) &\geq \int h_s \varphi d\mu \notag \\
    &= \ell \int h_s \mathbf{1}_{E} d\mu - \log \brackets{1 + (e^{\ell} - 1) \mu(E)} \int h_s d\mu \notag \\
    &\overset{(i)}{\geq} \ell \int h_s \mathbf{1}_{E} d\mu - \log \brackets{1 + (e^{\ell} - 1) \mu(E)} \notag \\
    &\overset{(ii)}{\geq} \ell \ \fA_{s} ((\ell, \ell+1]) - \log \brackets{1 + (e^{\ell} - 1) e^{-\ell} \fA_{s} ((\ell, \ell+1])} \notag \\ 
    &\overset{(iii)}{\geq} \ell \ \fA_{s} ((\ell, \ell+1]) -  \fA_{s} ((\ell, \ell+1]) \notag \\ 
    &\geq \frac{\ell}{2} \ \fA_{s} ((\ell, \ell+1]).
  \end{align}
  (i) follows from the definition of $h_s$ and $\int h_s d\mu \leq \int P_s f d\mu \leq 1$. (ii) follows from the fact that $h_s = P_s f$ on $E$, so $\mu(E) \leq e^{-\ell} \int P_s f \mathbf{1}_{E} d\mu = e^{-\ell} \fA_{s} ((\ell, \ell+1])$. (iii) follows from $\log(1+x) \leq x$ for all $x > -1$. The last line uses $\ell \geq 2$.

  Second, applying the log-Sobolev inequality on the Boolean hypercube (\cite{gross1975logarithmic} or Exercise 10.23 in~\cite{o2014analysis} for a comprehensive proof), we obtain
  \begin{align}
    \label{eq:app_log_sobolev}
    \Ent_\mu h_s &\leq 2 \dirichlet(\sqrt{h_s}, \sqrt{h_s})
  \end{align}
  for the Dirichlet form $\dirichlet$ defined in Eq.~\eqref{eq:def_Dirichlet_form}. 

  On the other hand, define the excess mass above the level $e^{u}$ at time $s$ as
  \begin{align*}
    F_s(u) \defn \int (P_s f - e^{u}) \mathbf{1}_{P_s f > e^{u}} d\mu.
  \end{align*}
  Since $s \mapsto P_s f$ is $C^1$ and the state space is finite, $s \mapsto F_s(u)$ is absolutely continuous for any fixed $u$. The chain rule for Lipschitz functions implies that for almost every $s$,
  \begin{align}
    \label{eq:derivative_F_s}
    \frac{d}{ds} F_s(u) = \int \fL P_s f \mathbf{1}_{P_s f > e^{u}} d\mu = - \dirichlet(\mathbf{1}_{P_s f > e^{u}}, P_s f).
  \end{align}
  Note that $\dirichlet(\mathbf{1}_{P_s f > e^{u}}, P_s f) = \frac14 \int \sum_i \Delta_i \mathbf{1}_{P_s f > e^{u}}  \Delta_i P_s f d\mu \geq 0$, because the sign of $\Delta_i \mathbf{1}_{P_s f > e^{u}}$ is the same as the sign of $\Delta_i P_s f$. 
  
  % It is the boundary measure of the level set $\braces{P_s f > e^{u}}$ under $\mu$, weighted by the finite difference $\abss{\Delta_i P_s f}$. It is also the rate at which the exceeding mass above the level $e^{u}$ decreases as $s$ increases.

  We claim that there exists a universal constant $c > 0$ such that
  \begin{align}
    \label{eq:claim_dirichlet_comparison}
    \dirichlet(\sqrt{h_s}, \sqrt{h_s}) \leq c \int_{\ell-1}^{\ell+2} \dirichlet(\mathbf{1}_{P_s f > e^{u}}, P_s f) du.
  \end{align}
  Deferring the proof of this claim, we are ready to conclude the proof of Lemma~\ref{lem:anticoncentration_profile_bound}. Combining Eq.~\eqref{eq:entropy_vs_anti_concentration_profile}, \eqref{eq:app_log_sobolev} and the claim~\eqref{eq:claim_dirichlet_comparison}, we obtain
  \begin{align*}
    \fA_{s} ((\ell, \ell+1]) &\leq \frac{2}{\ell} \Ent_\mu h_s \\
    &\leq \frac{4}{\ell} \dirichlet(\sqrt{h_s}, \sqrt{h_s}) \\
    &\leq \frac{c}{\ell} \int_{\ell-1}^{\ell+2} \dirichlet(\mathbf{1}_{P_s f > e^{u}}, P_s f) du \\
    &\leq - \frac{c}{\ell} \int_{\ell-1}^{\ell+2} \frac{d}{ds} F_s(u) du.
  \end{align*}
  Integrating $s$ from $0$ to $\infty$, applying Fubini's theorem and using the definition of $F_s(u)$, we obtain
  \begin{align*}
    \int_{0}^\infty \fA_{s} ((\ell, \ell+1]) ds &\leq \frac{c}{\ell} \int_{\ell-1}^{\ell+2} \brackets{F_0(u) - F_\infty(u)} du \leq \frac{3c}{\ell}.
  \end{align*}
  The last step uses the fact that $F_0(u) \leq \int P_0 f d\mu = \int f d\mu = 1$ and $F_{s}(u) \geq 0, \forall s, u$.

  \paragraph{Proof of the claim~\eqref{eq:claim_dirichlet_comparison}} The proof is a discrete version of the coarea formula. We have
  \begin{align*}
    \sqrt{h_s} = e^{\frac12 \log P_s f} \chi(\log P_s f - \ell) = \psi(\log P_s f),
  \end{align*}
  where we define $\psi(u) \defn e^{\frac12 u} \chi(u - \ell)$. We bound
  \begin{align*}
    \abss{\psi'(u)} = \abss{\frac12 e^{\frac12 u} \chi(u - \ell) + e^{\frac12 u} \chi'(u - \ell)} &\leq \parenth{\frac12 + c'} e^{\frac12 u} \mathbf{1}_{u \in (\ell - 1, \ell + 2)}.
  \end{align*}
  For $a, b \in \real$, define $[a, b)_* \defn [\min\braces{a, b}, \max\braces{a, b})$. Since $\psi(b) - \psi(a) = \int_a^b \psi'(u) du$, by Cauchy-Schwarz inequality, we have
  \begin{align*}
    \abss{\psi(b) - \psi(a)}^2 &= \abss{\int_a^b \psi'(u) du}^2 \\
    &\lesssim \abss{\int_a^b \mathbf{1}_{u \in (\ell - 1, \ell + 2)} du} \abss{\int_a^b e^{u} \mathbf{1}_{u \in (\ell - 1, \ell + 2)} du} \\
    &\lesssim \int_{\ell-1}^{\ell+2} \mathbf{1}_{\braces{u \in [a, b)_*} } du  \abss{e^b - e^a}.
  \end{align*}
  Applying the above inequality to $a = \log P_s f(x)$ and $b = \log P_s f(\flip_i (x))$, we obtain
  \begin{align*}
    \abss{\Delta_i \sqrt{h_s}(x)}^2 \lesssim \int_{\ell-1}^{\ell+2} \mathbf{1}_{\braces{e^u \in [P_s f(x), P_s f(\flip_i (x)))_*} } du \abss{\Delta_i P_s f(x)}.
  \end{align*}
  On the other hand, we have
  \begin{align*}
    \dirichlet(\mathbf{1}_{P_s f > e^{u}}, P_s f) & = \frac14 \sum_{i=1}^n \int \Delta_i \mathbf{1}_{P_s f > e^{u}} \Delta_i P_s f d\mu \\
    &= \frac14 \sum_{i=1}^n \int \mathbf{1}_{\braces{e^u \in [P_s f(x), P_s f(\flip_i (x)))_* }}  |\Delta_i P_s f(x)| d\mu.
  \end{align*}
  The last step follows because the sign of $\Delta_i \mathbf{1}_{P_s f > e^{u}}$ is the same as the sign of $\Delta_i P_s f$, which makes it nonnegative. We conclude by noting that $\dirichlet(\sqrt{h_s}, \sqrt{h_s}) = \frac14 \sum_{i=1}^n \int |\Delta_i \sqrt{h_s}(x)|^2 d\mu$ and comparing the two expressions above.
\end{proof}
\begin{remark}
  The proof of Lemma~\ref{lem:anticoncentration_profile_bound} can be summarized in one line as
  \begin{align*}
    \ell \int_0^\infty \fA_{s} ((\ell, \ell+1]) ds \lesssim  \int_0^\infty \Ent_\mu h_s ds  \lesssim \int_0^\infty \dirichlet(\sqrt{h_s}, \sqrt{h_s}) \lesssim \int_{\ell-1}^{\ell+2} \int_0^\infty \dirichlet(\mathbf{1}_{P_s f > e^{u}}, P_s f) ds du \lesssim 1.
  \end{align*}
  Vaguely speaking, the $1/\ell$ factor arises from translating the mass of the logarithmic band set (anti-concentration profile) to the entropy of a localized $P_s f$. Log-Sobolev inequality upper bounds the entropy by the localized Fisher information $\dirichlet(\sqrt{h_s}, \sqrt{h_s})$. The discrete coarea formula relates the localized Fisher information to the weighted boundary measure $\int_{\ell-1}^{\ell+2}\dirichlet(\mathbf{1}_{P_s f > e^{u}}, P_s f) du$, which is exactly the rate at which the heat semigroup sends mass above the level $e^{u}$. The last step uses the fact that the total mass sent across all level sets is bounded.

  The localized Fisher information and the weighted boundary measure are treated separately due to the lack of a chain rule in the discrete setting. Heuristically,
  \begin{align*}
    \dirichlet(\sqrt{h_s}, \sqrt{h_s}) &\approx \sum_i \frac{\abss{\Delta_i P_s f}^2}{P_s f } \mathbf{1}_{\log P_s f \approx \ell}, \\
    \int_{\ell-1}^{\ell+2} \dirichlet(\mathbf{1}_{P_s f > e^{u}}, P_s f) du &\approx \sum_i \abss{\Delta_i P_s f} \abss{\Delta_i \log P_s f} \mathbf{1}_{\log P_s f \approx \ell}.
  \end{align*}

  Finally, in the product function example $f(x) = \prod_{i=1}^n (1+x_i)$, $\fA_{s} ((\ell, \ell+1])$ can be as large as $c_s /\sqrt{\ell}$ with a constant $c_s$ that depends only on $s$. Lemma~\ref{lem:anticoncentration_profile_bound} shows that its time-smoothed version is of smaller order $1/\ell$. The main intuition is that the peak value of $c_s$ cannot persist for a long time.
\end{remark}

\section{Technical proofs of auxiliary lemmas}
In this section, we complete the missing proofs. 
\subsection{Proofs on the well-posedness of the joint reverse SDE}
\label{sub:proof_related_to_smoothness}
We prove Lemma~\ref{lem:edge_ratio} and Lemma~\ref{lem:uniqueness_existence_SDE_sol}.

\begin{proof}[Proof of Lemma~\ref{lem:edge_ratio}]
  Note that for $z \in \braces{-1, 1}^n$, we may write
  \begin{align*}
    f(z) = \sum_{y \in \braces{-1, 1}^n} f(y) \prod_{j=1}^n \parenth{\frac{1+y_j z_j}{2}}.
  \end{align*}
  This expression gives an alternative multilinear polynomial expansion, which takes the value $f(y)$ for any $y \in \braces{-1,1}^n$. By uniqueness of the multilinear polynomial expansion, we conclude that this polynomial agrees with the one in Eq.~\eqref{eq:multilinear_expansion}. It shows that for any $z \in [-1, 1]^n$, $f(z) \geq 0$,
  because $f(y) \geq 0$ for $y\in \braces{-1, 1}^n$ and $1+y_j z_j \geq 0$. Since $\sum_{y \in \braces{-1, 1}^n} \prod_{j=1}^n \parenth{\frac{1+y_j z_j}{2}} = 1$, if $f$ is $\braces{0,1}$-valued, then as a weighted average, $f(z) \in [0,1]$ for $z \in (-1,1)^n$. Additionally, for any $z\in (-1,1)^n$, because $1 - \abss{z_j} \leq 1+y_j z_j \leq 1 + \abss{z_j}$, 
  \begin{align*}
    \vecnorm{f}{1} \prod_{j=1}^n (1-\abss{z_j}) \leq f(z) \leq \vecnorm{f}{1} \prod_{j=1}^n (1+\abss{z_j}).
  \end{align*}

  Now for $z\in [-1, 1]^n$, write
  \begin{align*}
    f(z) &= (1+z_i) \underbrace{\sum_{y \in \braces{-1, 1}^n, y_i = 1} \frac12   f(y) \prod_{j \in [n], j \neq i} \parenth{\frac{1+y_j z_j}{2}} }_{=: R_1} \\
    &\qquad + (1-z_i) \underbrace{\sum_{y \in \braces{-1, 1}^n, y_i = -1} \frac12 f(y) \prod_{j \in [n], j \neq i} \parenth{\frac{1+y_j z_j}{2}} }_{=: R_2}. 
  \end{align*}
  Note that $R_1 \geq 0$ and $R_2 \geq 0$, for the same reason that $f(z) \geq 0$. We have
  \begin{align*}
    f(z) &= (1+z_i)R_1 + (1-z_i)R_2 \\
    f(\flip_i(z)) &=  (1-z_i)R_1 + (1+z_i)R_2
  \end{align*}
  Then
  \begin{align*}
    (1 - \abss{z_i}) (R_1 + R_2) &\leq f(z) \leq (1 + \abss{z_i}) (R_1 + R_2), \\
    (1 - \abss{z_i}) (R_1 + R_2) &\leq f(\flip_i(z)) \leq (1 + \abss{z_i}) (R_1 + R_2).
  \end{align*}
  Therefore,
  \begin{align*}
    \frac{1 - \abss{z_i}}{1 + \abss{z_i}} \leq \frac{f(\flip_i(z))}{f(z)} \leq \frac{1 + \abss{z_i}}{1 - \abss{z_i}}.
  \end{align*}
\end{proof}

\begin{proof}[Proof of Lemma~\ref{lem:uniqueness_existence_SDE_sol}]
  First, since a strictly positive $f$ is considered, the jump rate in Eq.~\eqref{eq:def_reverse_process} is always finite on a finite state space. Second, Lemma~\ref{lem:edge_ratio} gives
\begin{align}
  \label{eq:score_bound}
  \frac{-\abss{x_i}}{1-\abss{x_i}} \leq S_i(x) \leq  \frac{\abss{x_i}}{1+\abss{x_i}}, \quad \forall x \in (-1, 1)^n.
\end{align}
In particular, we obtain that for $t \in [0, \To]$, 
\begin{align*}
  \frac{1-e^{-\tau}}{1+e^{-\tau}} \leq 1 - 2 S_i(e^{-(T-t)} \uV_{t-}) \leq \frac{1+e^{-\tau}}{1-e^{-\tau}}.
\end{align*}
Back to Eq.~\eqref{eq:def_SDE_perturbed_reverse_process}, we verify that the jump rate on the $i$-th coordinate of $\uW_t$ is nonnegative and finite, as long as $0 \leq \bdelta < 1$. For $t\in [\theta,T]$, $x = e^{-(T-t)} \uV_{t-}$,
\begin{align*}
  \frac{1}{2} - S_i(x) + \delta_i \mathbf{1}_{t \leq \stopT} S_i(x) = 
  \begin{cases}
    \displaystyle \frac{1}{2} - S_i(x) + \bdelta \mathbf{1}_{t \leq \stopT} S_i(x) \geq \frac{1}{2} - S_i(x) \geq 0 & \text{ if } S_i(x) > 0, \\
    \displaystyle \parenth{\frac{1}{2} - S_i(x)} \brackets{\mathbf{1}_{t > \stopT} + \mathbf{1}_{t \leq \stopT}\frac{1}{1-2\bdelta S_i(x)} }  \geq 0 & \text{ otherwise}.
  \end{cases}
\end{align*}
Existence and uniqueness of the solution to the joint SDE~\eqref{eq:def_reverse_process} and~\eqref{eq:def_SDE_perturbed_reverse_process} follow from the strong solution theorem for stochastic equations driven by Poisson random measures \cite[Chapter XIV, Definition 14.5 and Theorem 14.23]{jacod1979calcul}. We verify the hypotheses of the strong solution theorem. On the finite-state space $\braces{-1,1}^{2n}$, the jump rate is bounded and predictable, because $\mathbf{1}_{t \leq \stopT}$ is left-continuous and adapted, hence predictable. The process satisfies the Lipschitz condition in the sense of~\cite[Definition 14.14]{jacod1979calcul} because on $\braces{-1,1}^{2n}$, if two histories of the process agree up to time $t$, then the jump rate at time $t$ is the same; otherwise their distance is at least $2$ and since the jump rate is bounded, the Lipschitz condition is trivially satisfied. The growth condition in the sense of~\cite[Definition 14.22]{jacod1979calcul} follows from the boundedness of the jump size and jump rate.

Additionally, as long as $0 \leq \bdelta < 1$, we have
\begin{align*}
  0 \leq \mathbf{1}_{S_i(x) > 0} + \frac{1-2 S_i(x)}{1 - 2 \bdelta S_i(x)} \mathbf{1}_{S_i(x) \leq 0} \leq \frac{1+e^{-\tau}}{1-e^{-\tau}}, \quad \forall x \in (-e^{-\tau}, e^{-\tau})^n,
\end{align*}
which leads to $0 \leq \delta_i \leq \btau \bdelta < \btau$.
\end{proof}

\subsection{Proofs related to Lemma~\ref{lem:coupling_TV_distance_control}}
\label{sub:proof_related_to_tv_lemma}
Here we prove Lemma~\ref{lem:level_1_ineq},~\ref{lem:expected_squared_score_bound},~\ref{lem:boolean_bridge_formula},~\ref{lem:Doob_h_transform} and~\ref{lem:weighted_energy_estimate}.

\begin{proof}[Proof of Lemma~\ref{lem:level_1_ineq}\label{proof_lem:level_1_ineq}]
  We first introduce the basics of $p$-biased analysis on the Boolean hypercube. For a reference, see~\cite{eldan2020concentration} or Section 8 of~\cite{o2014analysis}. Let $p = (p_1, \ldots, p_n) \in (0, 1)^n$, and let $\mu_p$ be the following product measure on $\braces{-1, 1}^n$, with
\begin{align*}
  \mu_p(y) \defn \prod_{i=1}^n \frac{1+y_i(2p_i - 1)}{2},
\end{align*}
where the $i$-th marginal takes value $1$ with probability $p_i$ and value $-1$ with probability $1-p_i$. For $i \in [n]$, define $\phi_i: \braces{-1, 1}^n \to \real$ as 
\begin{align}
  \label{eq:phi_basis_def}
  \phi_i(y) \defn \frac{1}{2} \parenth{\frac{1-2p_i}{\sqrt{p_i (1-p_i)}} + y_i \frac{1}{\sqrt{p_i (1-p_i)}} }.
\end{align}
For $S \subseteq [n]$, define $\phi_S \defn \prod_{i\in S} \phi_i$, and $\phi_\emptyset \defn 1$ by convention. 

\begin{proposition}[Proposition 31 in~\cite{eldan2020concentration}]
  \label{prop:p_biased_analysis}
  $(\phi_S)_{S \subseteq [n]}$ form an orthonormal basis for the space of functions on $\braces{-1, 1}^n$, equipped with the inner product $\angles{g, h}_{\mu_p} \defn \Exs_{\mu_p}[g h]$. Any function $h: \braces{-1, 1}^n \to \real$ can be written as 
  \begin{align*}
    h(y) = \sum_{S \subseteq [n]} \hat{h}_p(S) \phi_S(y), \quad \forall y \in \braces{-1, 1}^n,
  \end{align*}
  where $\hat{h}_p(S) \defn \Exs_{\mu_p}[h \phi_S]$ are called the $p$-biased Fourier coefficients of $h$. Additionally, let $\mathfrak{S} = \braces{i_1, \ldots, i_k} \subseteq [n]$ be a set of indices and $z \in (-1, 1)^n$, if $p = \frac{1+z}{2}$, then
  \begin{align}
    \label{eq:derivative_to_Fourier}
    \partial_{i_1} \ldots \partial_{i_k} h(z) = \parenth{\prod_{i\in \mathfrak{S}} \frac{1}{\sqrt{1-z_i^2}}} \hat{h}_p(\mathfrak{S}). 
  \end{align}
\end{proposition}
Note that Eq.~\eqref{eq:derivative_to_Fourier} differs from that in Proposition 31 in~\cite{eldan2020concentration} by a factor $4$ on each term in the product. For completeness, we record a proof of the identity we need in Appendix~\ref{proof:proof_of_prop1}.

Take $p = \frac{1+x}{2}$. Applying Parseval's inequality, we have
\begin{align}
  \label{eq:Parseval_h}
  \angles{h ,h}_{\mu_p} \geq \hat{h}_p(\emptyset)^2 + \sum_{i=1}^n \hat{h}_{p}(\braces{i})^2.
\end{align}
Note that $\angles{h ,h}_{\mu_p} = \Exs_{\mu_p} h^2 \overset{(i)}{=}  \Exs_{\mu_p} h = \hat{h}_p(\emptyset)$, where (i) follows from the fact that $h$ takes value in $\braces{0, 1}$. Additionally, according to Proposition~\ref{prop:p_biased_analysis}, 
\begin{align*}
  \hat{h}_p(\emptyset) = h(x), \ \hat{h}_p(\braces{i}) = \sqrt{1-x_i^2} \partial_i h(x).
\end{align*}
The result follows by substituting these values into Parseval's inequality~\eqref{eq:Parseval_h}. 
\end{proof}

\begin{proof}[Proof of Lemma~\ref{lem:expected_squared_score_bound}\label{proof_lem:expected_squared_score_bound}]
  It is more convenient to work with the scaled reverse process $\rV_t = \rho_t \uV_t$. Applying It\^o's formula to $\log f (\rV_t)$, we obtain 
  \begin{align*}
    d \log f(\rV_t) &= \sum_{i=1}^n \brackets{S_i(\rV_{t-}) dt + \int \log \parenth{1-2S_i(\rV_{t-})} \mathbf{1}_{0 < z \leq \frac12 - S_i(\rV_{t-})} \Ni(dt, dz) }\\
    &= \sum_{i=1}^n \brackets{S_i(\rV_{t-}) + \parenth{\frac12 - S_i(\rV_{t-})} \log \parenth{1-2S_i(\rV_{t-})}  }  dt \\
    &\quad \quad \quad + \sum_{i=1}^n \int \log \parenth{1-2S_i(\rV_{t-})} \mathbf{1}_{0 < z \leq \frac12 - S_i(\rV_{t-})} \wNi(dt, dz).
  \end{align*}
  We claim that 
  \begin{align}
    \label{eq:claim_lowerbound_xlogx}
    \frac{x\log(x) - x + 1}{(x-1)^2} \geq \frac{\log \btau}{2 (\btau - 1)} \geq \frac{1}{2\btau}, \quad \text{ for } \frac{1}{\btau} \leq x \leq \btau.
  \end{align}
  For $t \leq \To = T-\tau$, Lemma~\ref{lem:edge_ratio} implies that $\frac{1}{\btau} \leq 1-2 S_i(\rV_t) \leq \btau$. Applying the claim, by optional stopping theorem, integrating up to a bounded stopping time $\stopT$ and taking expectations yields
  \begin{align*}
    & \Exs \brackets{\log f(\rV_{\stopT}) \middle| \sF_\theta} \geq \log f(\rV_\theta) + \frac{\log \btau}{\btau - 1} \Exs\brackets{ \int_\theta^{\stopT} \sum_{i=1}^n S_i(\rV_{t-})^2 dt \middle| \sF_\theta}.
  \end{align*}
  The martingale term has expectation $0$ given $\sF_\theta$ since it has bounded jumps and finite quadratic variation on $[\theta, \To]$. If $\stopT = \theta$, the integral is $0$ and the estimate is trivial. If $\stopT > \theta$, then for $t<\stopT$, $\log f(\rV_{t}) \leq \log \eta + \frac12\log\log \eta + 1$ by the stopping time definition~\eqref{eq:def_stopT}. Almost surely, there is at most one jump at time $\stopT$ with size at most $\log(1-2S_i(\rV_{\stopT})) \leq \log \btau$. We conclude that
  \begin{align*}
    \Exs \brackets{\int_\theta^{\stopT} \sum_{i=1}^n S_i(\rV_{t-})^2 dt \middle| \sF_\theta} &\leq \frac{\btau - 1}{\log \btau} \parenth{R_\theta + \excess + \log \btau} \\
    &\leq \btau \parenth{R_\theta + \excess + 1}.
  \end{align*}

  \paragraph{Verify the claim~\eqref{eq:claim_lowerbound_xlogx}} Consider the function $q(x) \defn \frac{x\log(x) - x + 1}{(x-1)^2}$. Differentiating to obtain
  \begin{align*}
    q'(x) = \frac{x+1}{(x-1)^2} \brackets{ \frac{2}{x+1} - \frac{\log x}{x-1}} = \frac{x+1}{(x-1)^2} \brackets{ \frac{2}{x+1} - \frac{\int_1^x \frac{1}{t} dt}{x-1}}.
  \end{align*}
  Since $1/t$ is convex, for $x > 1$, by Jensen's inequality, $\frac{1}{x-1}\int_1^x \frac{1}{t} dt \geq \frac{1}{\frac{x+1}{2}} = \frac{2}{x+1}$. For $x < 1$, $\frac{1}{1-x}\int_x^1 \frac{1}{t} dt \geq \frac{1}{\frac{x+1}{2}} = \frac{2}{x+1}$. $q$ is non-increasing on $(0, \btau)$. Hence, $q(x) \geq q(\btau)$ for $\frac{1}{\btau} \leq x \leq \btau$.
  Note that
  \begin{align*}
    q(\btau) = \frac{1}{(\btau-1)^2} \int_1^{\btau} \log t dt
  \end{align*}
  Since $\log t$ is strictly concave, the integral $\int_1^{\btau} \log t dt$ is lower bounded by the area of the triangle formed by the points $(1, 0), (\btau, \log \btau)$ and the $x$-axis, which is $\frac{\log \btau}{2} (\btau - 1)$. Hence,
  \begin{align*}
    q(\btau) \geq \frac{\log \btau}{2 (\btau - 1)} \geq \frac{1}{2 \btau}.
  \end{align*}
\end{proof}

\begin{proof}[Proof of Lemma~\ref{lem:boolean_bridge_formula}\label{proof_lem:boolean_bridge_formula}]
  Let $K_t$ denote the heat kernel
  \begin{align}
    \label{eq:heat_kernel}
    K_t(x, y) \defn \Prob(U_t = y \mid U_0 = x) = \prod_{i=1}^n \parenth{\frac{1+e^{-t} x_i y_i}{2}}, \quad x, y \in \braces{-1, 1}^n.
  \end{align}
  Conditioned on $\uV_t = x$ and $\uV_T = \zeta$, for $t \leq \To \leq T$, the law of $\uV_\To$ is the heat bridge. That is, by Bayes' rule and by identifying the law of $(\uV_t)_{t \in [0, T]}$ with that of $(U_{T-t})_{t \in [0, T]}$, we have
  \begin{align*}
    \Prob(\uV_{\To} = z \mid \uV_t = x, \uV_{T} = \zeta) &= \Prob(U_{\tau} = z \mid U_{T-t} = x, U_0 = \zeta) \\
    &= \frac{\Prob(U_{\tau} = z, U_{T-t} = x \mid U_0 = \zeta)}{\Prob(U_{T-t} = x \mid U_0 = \zeta)} \\
    &= \frac{\Prob(U_{T-t} = x \mid U_{\tau} = z, U_0 = \zeta) \Prob(U_{\tau} = z \mid U_0 = \zeta)}{\Prob(U_{T-t} = x \mid U_0 = \zeta)} \\
    &\overset{(i)}{=} \frac{K_{\To-t}(z, x) K_{\tau}(\zeta, z)}{K_{T-t}(\zeta, x)},
\end{align*}
  where (i) follows from Markov property of $U_t$. In particular, this implies that the coordinates are conditionally independent, and the one-coordinate law is
  \begin{align*}
    \Prob\parenth{\uV_{\To}^{[i]} = z_i \mid \uV_t = x, \uV_{T} = \zeta} =  \frac{ \frac{1+ \rho_t/\rho_{\To} z_i x_i }{2} \cdot \frac{1 + \rho_\To z_i \zeta_i}{2}  }{\frac{1+\rho_t x_i \zeta_i}{2}}.
  \end{align*}
  Its expectation takes the form
  \begin{align*}
    \Exs\brackets{\uV_{\To}^{[i]} \mid \uV_t = x, \uV_{T} = \zeta} &= \frac{\rho_\To \zeta_i + \rho_t/\rho_{\To} x_i}{1+\rho_t x_i \zeta_i}
  \end{align*}
  Multiplying coordinate-wise by $x_i y_i$, we obtain
  \begin{align*}
    m_{t}^{[i]}(x, y, \zeta) = \frac{\rho_t/\rho_{\To} y_i + \rho_\To \sigma_i}{1+\rho_t x_i \zeta_i},
  \end{align*}
  where $\sigma_i = x_i y_i \zeta_i$.
  Using the fact that $\frac{1}{1+\rho_t x_i \zeta_i} = \frac{1 - \rho_t x_i \zeta_i}{1 - \rho_t^2}$, we simplify the denominator and obtain the desired expression with $a_{t}$ and $b_{t}$ defined in Eq.~\eqref{eq:def_a_b} as follows
  \begin{align*}
      m_{t, \To}^{[i]}(x, y, \zeta) = \frac{ (\rho_t/\rho_{\To} - \rho_t \rho_\To) y_i + (\rho_\To - \rho_t^2 /\rho_\To) \sigma_i }{1-\rho_t^2} = a_{t} y_i + b_{t} \sigma_i. 
  \end{align*} 
  We have identified the coefficients
  \begin{align*}
    a_{t} &= \frac{1/\rho_{\To} - \rho_\To}{1/\rho_t-\rho_t} = \frac{\sinh(T-\To)}{\sinh(T-t)},\\
    b_{t} &= \frac{\rho_\To/\rho_t - \rho_t /\rho_\To}{1/\rho_t-\rho_t} = \frac{\sinh(\To-t)}{\sinh(T-t)}.
  \end{align*}
  
  For the derivatives, note that flipping $y_i$ changes the $i$-th coordinate of $m_t$ from $a_t y_i + b_t \sigma_i$ to $- (a_t y_i + b_t \sigma_i)$, while keeping the other coordinates unchanged. By multilinearity of $\phi$, we have
  \begin{align*}
    \Delta_i^y q_t^{\zeta}(x, y) &= - 2 (a_t y_i + b_t \sigma_i) \partial_i \phi(m_t(x, y, \zeta)).
  \end{align*}
  Similarly, when we apply it to $(\flip_i(x), y)$, only $\sigma_i$ changes sign and $\partial_i \phi$ is independent of the $i$-coordinate of its input because $\phi$ is multilinear. Hence,
  \begin{align*}
    \Delta_i^y q_t^{\zeta}(\flip_i(x), y) &= - 2 (a_t y_i - b_t \sigma_i) \partial_i \phi(m_t(x, y, \zeta)).
  \end{align*}
  A synchronized flip keeps $\sigma_i$ unchanged, hence, 
  \begin{align*}
    \Delta_i^{xy} q_t^{\zeta}(x, y) &= - 2 a_t y_i \partial_i \phi(m_t(x, y, \zeta)).
  \end{align*}
\end{proof}

\begin{proof}[Proof of Lemma~\ref{lem:Doob_h_transform}] 
  Let $Z_t \defn (\uV_t, \uW_t)$. The proof intuitively works like the usual proof of Doob's $h$-transform for a Markov process (see e.g., Section 7.5~\cite{sarkka2019applied}). However, since $Z_t$ is not Markov but a pure-jump process with predictable jump rates, we need to use the change of measures formula in~\cite[Chapter III]{jacod2013limit}.

  First, fix $\zeta \in \braces{-1, 1}^n$. Since $f > 0$, $\Prob(\uV_T = \zeta) > 0$, so $\Prob^\zeta \defn \Prob(\cdot \mid \uV_{T} = \zeta)$ is well-defined. Observe that 
  \begin{align*}
    \Prob(\uV_T = \zeta \mid \sF_t) \overset{(i)}{=} \Prob(\uV_T = \zeta \mid \uV_t) = H_t^{\zeta}(\uV_t).
  \end{align*}
  (i) holds because the $\uV$-marginal is the unperturbed reverse heat process and its future evolution does not depend on $\uW$, $\bdelta$ or the stopping rule. Define the density process of $\Prob^\zeta$ with respect to $\Prob$ as
  \begin{align*}
    D_t^{\zeta} \defn \frac{H_t^{\zeta}(\uV_t)}{\Prob(\uV_T = \zeta)}. 
  \end{align*}
  Then the post-jump and pre-jump density ratio on the $i$-th coordinate is
  \begin{align*}
    \frac{D_t^{\zeta}}{D_{t-}^{\zeta}} = \frac{H_t^{\zeta}(\uV_t)}{H_{t-}^{\zeta}(\uV_{t-})} = 1 + \parenth{r^{\zeta}_{t, i}(\uV_{t-}) - 1} \mathbf{1}_{0 < z \leq \frac12 - S_i(\rho_t \uV_{t-})},
  \end{align*}
  where $r^{\zeta}_{t, i}(x) = \frac{H_t^{\zeta}(\flip_i(x))}{H_{t}^{\zeta}(x)}$.

  Second, apply the change of measures formula~\cite[Chapter III, Theorem 3.17]{jacod2013limit} to $\Ni$ with the density process $D_t^{\zeta}$ and the original compensator $dtdz$. Then
  \begin{align*}
    \nu^{\zeta, [i]}(dt, dz) = \brackets{1 + \parenth{r^{\zeta}_{t, i}(\uV_{t-}) - 1} \mathbf{1}_{0 < z \leq \frac12 - S_i(\rho_t \uV_{t-})}} dt dz,
  \end{align*}
  is the $\Prob^{\zeta}$-predictable compensator of $\Ni(dt, dz)$. 
  In words, if there is a flip on the $i$-th coordinate of $\uV$ at time $t$, then the jump rate is multiplied by $r^{\zeta}_{t, i}(\uV_{t-})$ under $\Prob^\zeta$, otherwise it remains the same. 

  Third, let $h: \braces{-1, 1}^{2n} \to \real$. Introduce the shorthand
  \begin{align*}
    S_i = S_i(\rho_t x),\ \delta_i = \delta_i(\rho_t x),\ r_i^{\zeta} = r^{\zeta}_{t, i}(x), \\
    \lambda_i^{V} = \frac12 - S_i,\ \lambda_i^{W} = \frac12 - S_i + \mathbf{1}_{t \leq \stopT} \delta_i S_i.
  \end{align*}
  By the jump SDEs for $(\uV, \uW)$, 
  \begin{align*}
    h(Z_t) - h(Z_\theta) = \int_\theta^t \sum_{i=1}^n \int_\real \Phi_{s, i}^{h}(Z_{s-}, z) \Ni(ds, dz),
  \end{align*}
  where 
  \begin{align*}
    \Phi_{s, i}^{h}(Z_{s-}, z) &\defn \Delta_i^{xy} h(Z_{s-}) \mathbf{1}_{0 < z < (\lambda_i^V \wedge \lambda_i^W)} \\
    &\quad + \Delta_i^x h(Z_{s-}) \mathbf{1}_{\lambda_i^W < z \leq \lambda_i^V} \\
    &\quad + \Delta_i^y h(Z_{s-}) \mathbf{1}_{\lambda_i^V < z \leq \lambda_i^W}.
  \end{align*}
  By the definition of the predictable compensator, the compensated integral
  \begin{align*}
    \int_\theta^t \sum_{i=1}^n \int_\real \Phi_{s, i}^{h}(Z_{s-}, z) \parenth{\Ni(ds, dz)  - \nu^{\zeta, [i]}(ds, dz)}
  \end{align*}
  is a $\Prob^\zeta$-martingale. Hence, the joint predictable generator of $Z_t$ under $\Prob^\zeta$ is
  \begin{align*}
    \cG_t^{\zeta} h(x, y) &= \sum_{i=1}^n r_i^{\zeta} \parenth{\lambda_i^V \wedge \lambda_i^W} \Delta_i^{xy} h (x, y) + \sum_{i=1}^n r_i^{\zeta}  \parenth{\lambda_i^V - \lambda_i^W}_+ \Delta_i^x h (x, y) \\
    &\quad + \sum_{i=1}^n \parenth{\lambda_i^W - \lambda_i^V}_+ \Delta_i^y h (x, y).
  \end{align*}
  The factor $r_i^{\zeta}$ appears exactly on jumps in which $\uV$ flips. 
  Since $\lambda_i^W - \lambda_i^V = \mathbf{1}_{t \leq \stopT} \delta_i S_i$, and using $\Delta_i^y h(\flip_i(x), y) = \Delta_i^{xy} h(x, y) - \Delta_i^x h(x, y)$, we can rewrite the generator as
  \begin{align*}
    \cG_t^{\zeta} &= \jL_t^{0, \zeta} + \mathbf{1}_{t \leq \stopT} A_t^{\zeta},
  \end{align*}
  where
  \begin{align*}
    \jL_t^{0, \zeta} h(x, y) &= \sum_{i=1}^n r_{i}^{\zeta} \parenth{\frac12 - S_i } \Delta_i^{xy} h(x, y), \\
    A_t^\zeta h(x, y) &= \sum_{i=1}^n  \mathbf{1}_{S_i > 0}  \delta_i  S_i \Delta_{i}^y h(x, y)  + \mathbf{1}_{S_i \leq 0} r_{i}^{\zeta} \delta_i  S_i  \Delta_i^y h(\flip_i(x), y).
  \end{align*}

  Finally, to simplify the conditioned jump rate, using the heat kernel in Eq.~\eqref{eq:heat_kernel} and Bayes' rule, we have
  \begin{align*}
    H_t^{\zeta}(x) = \Prob(\uV_T = \zeta \mid \uV_t = x) = \frac{\Prob(\uV_t = x \mid \uV_T = \zeta) \Prob(\uV_T = \zeta)}{\Prob(\uV_t = x)} = \frac{K_{T-t}(\zeta, x) f(\zeta)}{P_{T-t} f(x)}.
  \end{align*}
  Then 
  \begin{align*}
    r_i^{\zeta} \parenth{\frac12 - S_i } &= \frac{K_{T-t}(\zeta, \flip_i(x)) f(\zeta) }{P_{T-t} f(\flip_i(x))} \cdot \frac{P_{T-t} f(x)}{K_{T-t}(\zeta, x) f(\zeta)} \cdot \frac12 \frac{P_{T-t} f(\flip_i(x))}{P_{T-t} f(x)}  \\
    &= \frac12 \frac{K_{T-t}(\zeta, \flip_i(x)) }{K_{T-t}(\zeta, x)} \\
    &= \frac12 \frac{1 - \rho_t x_i \zeta_i}{1 + \rho_t x_i \zeta_i}.
  \end{align*}

\end{proof}

\begin{proof}[Proof of Lemma~\ref{lem:weighted_energy_estimate}]
  First, we bound $\Psi_b$. For $t \in [\theta, \To]$, $x, y, \zeta \in \braces{-1, 1}^n$ and $i \in [n]$, the definition of $m_t$ in Lemma~\ref{lem:boolean_bridge_formula} yields 
  \begin{align*}
    1- m_t^{[i]}(x, y, \zeta)^2 &= \frac{(1-\rho_{\To}^2)(\rho_{\To}^2 - \rho_t^2)} {\rho_{\To}^2 (1+\rho_t x_i \zeta_i)^2} \\
    b_t^2 &= \frac{\sinh^2(\To-t)}{\sinh^2(T-t)} = \frac{(\rho_{\To}^2 - \rho_t^2)^2}{\rho_{\To}^2 (1-\rho_t^2)^2}.
  \end{align*}
  Taking the ratio and multiplying by $\lambda_{t, i}^{\zeta}(x)$ yields, for $t < \To$,
  \begin{align*}
    \frac{\lambda_{t, i}^{\zeta}(x) b_t^2}{1- m_t^{[i]}(x, y, \zeta)^2} = \frac{ \rho_{\To}^2 - \rho_t^2}{1-\rho_t^2} \cdot \frac{1}{1-\rho_{\To}^2} \leq \frac{\rho_{\To}^2}{1-\rho_{\To}^2} = \frac{e^{-2\tau}}{1-e^{-2\tau}}.
  \end{align*}
  Hence,
  \begin{align*}
    \lambda_{t, i}^{\zeta}(x) b_t^2 \abss{\partial_i \phi(m_t(x, y, \zeta))}^2 &\leq \frac{e^{-2\tau}}{1-e^{-2\tau}} \parenth{1- m_t^{[i]}(x, y, \zeta)^2} \abss{\partial_i \phi(m_t(x, y, \zeta))}^2.
  \end{align*}
  Summing over $i$, applying the level-1 inequality in Lemma~\ref{lem:level_1_ineq} and integrating, we conclude that $\Psi_b \leq \frac{e^{-2\tau}}{1-e^{-2\tau}} (\To - \theta)$.

  Second, we bound $\Psi_a$. Fix $\zeta \in \braces{-1,1}^n$ and work under the conditioned law $\Prob^\zeta$. From $m_t$ and $q_t^{\zeta}$ in Lemma~\ref{lem:boolean_bridge_formula}, we have
  \begin{align*}
    \partial_t m_t^{[i]}(x, y, \zeta) &= \partial_t a_t y_i + \partial_t b_t \sigma_i = a_t y_i \frac{1-\rho_t x_i \zeta_i}{1+\rho_t x_i \zeta_i} \\
    \Delta_i^{xy} q_t^{\zeta}(x, y) & = - 2 a_t y_i \partial_i \phi(m_t(x, y, \zeta)).
  \end{align*}
  Using the multilinearity of $\phi$ and the unperturbed generator representation in Lemma~\ref{lem:Doob_h_transform}, we verify that
  \begin{align}
    \label{eq:q_zeta_harmonic}
    \parenth{\partial_t + \jL_t^{0, \zeta}} q_t^{\zeta} = 0.
  \end{align}
  Using the conditioned predictable joint generator from
  Lemma~\ref{lem:Doob_h_transform}, apply It\^o's formula to $q_{t}^{\zeta}(\uV_t, \uW_t)^2$ under the conditioned law $\Prob^\zeta$ to obtain
  \begin{align}
    \label{eq:q_zeta_sq_ito}
    \Exs^{\zeta} \brackets{q_{\To}^{\zeta}(\uV_{\To}, \uW_{\To})^2 - q_{\theta}^{\zeta}(\uV_\theta, \uW_\theta)^2} &= \Exs^{\zeta} \int_\theta^{\To} \parenth{\partial_t + \jL_t^{0, \zeta} + \mathbf{1}_{t \leq \stopT} A_t^{\zeta}} \parenth{q_{t}^{\zeta}}^2 (\uV_{t-}, \uW_{t-}) dt.
  \end{align}
  For the unperturbed part, Eq.~\eqref{eq:q_zeta_harmonic} yields
  \begin{align*}
    \parenth{\partial_t + \jL_t^{0, \zeta}} \parenth{q_{t}^{\zeta}}^2 (x, y) &= \sum_{i=1}^n \frac12 \lambda_{t, i}^{\zeta}(x) \parenth{\Delta_i^{xy} q_{t}^{\zeta} (x, y)}^2 \\
    &= 2 a_t^2\sum_{i} \lambda_{t, i}^{\zeta}(x) \abss{\partial_i \phi(m_t(x, y, \zeta))}^2.
  \end{align*}
  For the perturbation part, recall that 
  \begin{align*}
    A_t^\zeta h(x, y) &= \sum_{i=1}^n  \mathbf{1}_{S_i > 0}  \delta_i  S_i \Delta_{i}^y h(x, y)  + \mathbf{1}_{S_i \leq 0} r_{i}^{\zeta} \delta_i  S_i  \Delta_i^y h(\flip_i(x), y).
  \end{align*}
  Note that from Lemma~\ref{lem:Doob_h_transform} and $\btau^{-1} \leq \lambda_{t, i}^{\zeta}(x) \leq \btau$,
  \begin{align*}
    \abss{\delta_i \brackets{\mathbf{1}_{S_i > 0} + r_i^{\zeta} \mathbf{1}_{S_i \leq 0} }}^2 = \abss{\bdelta \brackets{\mathbf{1}_{S_i > 0} + \frac{\lambda_{t, i}^{\zeta}(x)}{1 - 2 \bdelta S_i} \mathbf{1}_{S_i \leq 0}}}^2
    \leq \bdelta^2 \btau  \lambda_{t, i}^{\zeta}(x).
  \end{align*}
  Additionally, since $q^{\zeta}$ is bounded by $1$ because $\phi$ as a multilinear extension takes value in $[0, 1]$, we have
  \begin{align*}
    \abss{\Delta_i^y (q_{t}^{\zeta})^2 (x, y)} &\leq 2 \abss{\Delta_i^y q_{t}^{\zeta} (x, y)} \\
    \abss{\Delta_i^y (q_{t}^{\zeta})^2 (\flip_i(x), y)} &\leq 2 \abss{\Delta_i^y q_{t}^{\zeta} (\flip_i(x), y)}.
  \end{align*}
  Hence,
  \begin{align*}
    \abss{A_t^{\zeta} (q_{t}^{\zeta})^2 (x, y)} &\lesssim \bdelta \btau^{\frac12} \sum_{i=1}^n \parenth{\lambda_{t, i}^{\zeta}(x)}^{\frac12} \abss{S_i} \parenth{\abss{\Delta_i^y q_{t}^{\zeta} (x, y)} + \abss{\Delta_i^y q_{t}^{\zeta} (\flip_i(x), y)}} \\
    &\overset{(i)}{\lesssim} \btau^{\frac12} \parenth{\sum_{i=1}^n \bdelta^2 S_i^2}^{\frac12} \parenth{\sum_{i=1}^n \lambda_{t, i}^{\zeta}(x) \parenth{\abss{\Delta_i^y q_{t}^{\zeta} (x, y)}^2 + \abss{\Delta_i^y q_{t}^{\zeta} (\flip_i(x), y)}^2}}^{\frac12} \\
    &\overset{(ii)}{\lesssim} \btau^{\frac12} \parenth{\sum_{i=1}^n \bdelta^2 S_i^2}^{\frac12} \parenth{\sum_{i=1}^n \lambda_{t, i}^{\zeta}(x) \parenth{a_t^2 + b_t^2}\abss{\partial_i \phi(m_t(x, y, \zeta))}^2 }^{\frac12}
  \end{align*}
  where (i) follows from Cauchy-Schwarz inequality, and (ii) follows from $\Delta_i^y q_{t}^{\zeta}$ expression in Lemma~\ref{lem:boolean_bridge_formula}. 
  Evaluating at $(\uV_{t-}, \uW_{t-})$, integrating in time, applying Cauchy-Schwarz, averaging over $\zeta = \uV_T$ and plugging back in Eq.~\eqref{eq:q_zeta_sq_ito} yields
  \begin{align*}
    \Psi_a &\lesssim 1 + \btau^{\frac12} \cS^{\frac12} \parenth{\Psi_a + \Psi_b}^{\frac12} \\
    &\lesssim 1 + \btau \cS + \frac14 (\Psi_a + \Psi_b).
  \end{align*}
  Absorbing $\frac14 \Psi_a$ to the left-hand side yields the desired bound
  \begin{align*}
    \Psi_a &\lesssim 1 + \btau \cS + \Psi_b.
  \end{align*}
\end{proof}

\subsection{Proofs related to Lemma~\ref{lem:approx_monotone_coupling}}
\label{sub:proof_related_to_monotone_coupling}

\begin{proof}[Proof of Lemma~\ref{lem:logfV_lower_bound}\label{proof_lem:logfV_lower_bound}]
  Recall from Eq.~\eqref{eq:logf_VW_simple_notation} that $\log f(\rV_t)$ satisfies the following SDE
  \begin{align*}
    d \log f(\rV_t) &= \sum_{i=1}^n \brackets{v_i dt + \int \log \parenth{1-2v_i} \mathbf{1}_{0 < z \leq \frac12 - v_i} \Ni(dt, dz) },
  \end{align*}
  where $v_i = S_i(\rV_{t-})$. Let $\psiV_i \defn \log \parenth{1-2v_i} \mathbf{1}_{0 < z \leq \frac12 - v_i}$. Let $\gamma \in \real$.  Introduce the exponentiated process
  \begin{align*}
    \PsiV_t \defn \exp\brackets{\gamma \int_0^t \int \sum_i \psiV_i \Ni(ds, dz) - \int_0^t \int \sum_i \brackets{\exp(\gamma \psiV_i) - 1} dz ds }.
  \end{align*}
  Applying It\^o's formula, we obtain
  \begin{align*}
    d \PsiV_t &= \PsiV_{t-} \parenth{- \int \sum_{i=1}^n \brackets{\exp(\gamma \psiV_i) - 1} dz dt} + \sum_{i=1}^n \int \brackets{\PsiV_{t-} \exp(\gamma \psiV_i) - \PsiV_{t-}} \Ni(dt, dz) \\
    &= \sum_{i=1}^n \int \brackets{\PsiV_{t-} \exp(\gamma \psiV_i) - \PsiV_{t-}} \wNi(dt, dz).
  \end{align*}
  Take $\gamma = -1$, observe that $\PsiV_t = \frac{f(\rV_0)}{f(\rV_t)}$. Together with Lemma~\ref{lem:edge_ratio} which provides a lower bound on $f(\rV_t)$ for $t \leq \To$, $(\PsiV_t)$ is a bounded local martingale on $[0, \To]$. The optional stopping theorem gives $\Exs[\PsiV_{\To} \mid \sF_{\stopT}] = \PsiV_{\stopT}$. 
  Apply Markov's inequality conditioned on $\sF_\stopT$, for any $\beta > 0$, we have
  \begin{align*}
    \Prob(\PsiV_{\To} \geq \beta \PsiV_{\stopT} \mid \sF_{\stopT}) \leq \frac{\Exs[\PsiV_{\To} \mid \sF_{\stopT}]}{\beta \PsiV_{\stopT}}  \leq \frac{1}{\beta}. 
  \end{align*}
  Taking expectation gives
  \begin{align*}
    \Prob(- \log f(\rV_{\To})  \geq - \log f(\rV_{\stopT}) + \log \beta) = \Prob(\PsiV_{\To} \geq \beta \PsiV_{\stopT}) \leq \frac{1}{\beta}. 
  \end{align*}
  Conclude with the choice $\beta = \sqrt{\log \eta}$.
\end{proof}

\begin{proof}[Proof of Lemma~\ref{lem:wivi_bound}\label{proof_lem:wivi_bound}]
  Recall from Eq.~\eqref{eq:logf_VW_simple_notation} that $\log f(\rW_t)$ satisfies the following SDE
  \begin{align*}
    d \log f(\rW_t) &= \sum_{i=1}^n \brackets{w_i dt + \int \log \parenth{1-2w_i} \mathbf{1}_{0 < z \leq \frac12 - (1 - \delta_i \mathbf{1}_{t \leq \stopT})v_i} \Ni(dt, dz)},
  \end{align*}
  where $(v_i, w_i) = (S_i(\rV_{t-}), S_i(\rW_{t-}))$. Let $\psiW_i \defn \log \parenth{1-2w_i} \mathbf{1}_{0 < z \leq \frac12 - (1 - \delta_i \mathbf{1}_{t \leq \stopT})v_i}$. Let $\gamma \in \real$. Introduce the exponentiated process for $t \in [\theta, \To]$,
  \begin{align*}
    \PsiW_t \defn \exp\brackets{\gamma \int_\theta^t \int \sum_i \psiW_i \Ni(ds, dz) - \int_\theta^t \int \sum_i \brackets{\exp(\gamma \psiW_i) - 1} dz ds }.
  \end{align*}
  Apply It\^o's formula, similar to the calculation in Lemma~\ref{lem:logfV_lower_bound}, deduce that $(\PsiW_t)$ is a positive local martingale on $[\theta, \To]$. It satisfies $\Exs[\PsiW_{\stopT}] \leq \Exs[\PsiW_\theta]= 1$. Markov's inequality gives that for any $\beta > 0$,
  \begin{align*}
    \Prob(\log \PsiW_\stopT \geq \log \beta) = \Prob(\PsiW_\stopT \geq \beta) \leq \frac{\Exs[\PsiW_\stopT]}{\beta} \leq \frac{1}{\beta}.
  \end{align*}
  Take $\beta = \sqrt{\log \eta}$ and $\gamma = 1$. Notice that
  \begin{align*}
    \int [\exp(\gamma \psiW_i) - 1 ] dz &= (-2w_i) \parenth{\frac12 - (1-\delta_i \mathbf{1}_{t \leq \stopT}) v_i} \\
    &= - w_i + 2 (1-\delta_i \mathbf{1}_{t \leq \stopT}) w_i v_i.
  \end{align*}
  Combining with the definition of $\PsiW_t$, we obtain
  \begin{align*}
    \log \PsiW_\stopT = \log f (\rW_\stopT) - \log f(\rW_\theta) - 2 \sum_{i=1}^n \int_\theta^\stopT (1-\delta_i) w_i v_i dt.
  \end{align*}
  Therefore, the above Markov's inequality states that with probability at least $1 - \frac{1}{\sqrt{\log \eta}}$,
  \begin{align*}
    \log f(\rW_{\stopT}) -  \log f (\rW_{\theta}) < \brackets{2 \sum_{i=1}^n \int_\theta^\stopT   (1-\delta_i) w_i v_i dt } + \frac{1}{2} \log\log\eta.
  \end{align*}
  
\end{proof}

\begin{proof}[Proof of Lemma~\ref{lem:log_diff_bound}\label{proof_lem:log_diff_bound}]
  From SDEs~\eqref{eq:logf_VW_simple_notation} for $\log f(\rW_t)$ and $\log f(\rV_t)$, take difference to obtain
  \begin{align*}
    d [\log f(\rW_t) - \log f(\rV_t)]  = \sum_{i=1}^n [w_i - v_i] dt + \sum_{i=1}^n \int \psi_i \Ni(dt, dz),
  \end{align*}
  where $\psi_i \equiv \psi_i(\rV_{t-}, \rW_{t-}, z)$ is defined as follows, with $I_i \defn \min\braces{\frac12 - v_i, \frac12 - v_i + \delta_i \mathbf{1}_{t \leq \stopT} v_i}$, 
  \begin{align*}
    \psi_i \defn \begin{cases}
      [\log (1-2w_i) - \log(1-2v_i)] \mathbf{1}_{0 < z \leq I_i} + \log(1-2w_i) \mathbf{1}_{I_i < z < \frac12 - v_i + \delta_i \mathbf{1}_{t \leq \stopT} v_i} & \text{ if } v_i > 0 \\
      [\log (1-2w_i) - \log(1-2v_i)] \mathbf{1}_{0 < z \leq I_i} - \log(1-2v_i) \mathbf{1}_{I_i < z < \frac12 - v_i} & \text{ otherwise }.
    \end{cases}
  \end{align*}
  In words, there are extra jumps on $\rW$ or $\rV$ depending on whether $v_i > 0$ or not.  
  Let $\gamma \in \real$. Introduce the exponentiated process for $t \in [\theta, \To]$,
  \begin{align}
    \label{eq:exponentiated_process_lem11}
    \Psi_t \defn \exp\brackets{\gamma \int_\theta^t \int \sum_i \psi_i \Ni(ds, dz) - \int_\theta^t \int \sum_i \brackets{\exp(\gamma \psi_i) - 1} dz ds }.
  \end{align}
  Apply It\^o's formula, similar to that in Lemma~\ref{lem:logfV_lower_bound}, deduce that $(\Psi_t)$ is positive local martingale on $[\theta, \To]$. It satisfies $\Exs[\Psi_{\To}] \leq \Exs[\Psi_\theta] = 1$. Markov's inequality gives that for any $\beta > 0$,
  \begin{align*}
    \Prob(\log \Psi_\To \geq \log \beta) = \Prob(\Psi_\To \geq \beta) \leq \frac{\Exs[\Psi_\To]}{\beta} \leq \frac{1}{\beta}.
  \end{align*}
  Take $\beta = \sqrt{\log \eta}$ and $\gamma = 1$. Notice that
  \begin{align*}
    \int [\exp(\psi_i) - 1 ] dz = \begin{cases}
      \displaystyle -[w_i - v_i] - 2 \delta_i \mathbf{1}_{t \leq \stopT} w_i v_i & \text{ if } v_i > 0 \\
      \displaystyle -[w_i - v_i] - \frac{2 \delta_i \mathbf{1}_{t \leq \stopT} w_i v_i}{1-2v_i} & \text{ otherwise. } 
    \end{cases}
  \end{align*}
  Therefore, with probability at least $1 - \frac{1}{\sqrt{\log \eta}}$,
  \begin{align*}
    \log f(\rW_\To) - \log f(\rV_\To) < \brackets{- 2 \sum_{i=1}^n \int_\theta^{\stopT} \delta_i w_i v_i \parenth{\mathbf{1}_{v_i > 0} + \mathbf{1}_{v_i \leq 0} \frac{1}{1-2v_i} } dt } + \frac12 \log\log\eta. 
  \end{align*}

\end{proof}

\section*{Acknowledgements}
The author thanks Ronen Eldan for patiently answering questions through numerous email exchanges, for generously sharing research notes and for constant encouragement. He thanks Dan Mikulincer for introducing several Boolean problems, as well as Galen Reeves and Bo'az Klartag for insightful discussions. He also thanks Joseph Lehec for kindly pointing out a mistake in the previous draft. The author is grateful for the Simons Institute for the Theory of Computing, especially the 2023 workshop ``Beyond the Boolean Cube'' for lectures on the basics of Boolean function analysis and for providing an excellent environment for research exchange.  

%%%%%%%%%%%% DO BIBLIOGRAPHY %%%%%%%%%%%%%%%%%%%%%%%%%%%%%%%%%%%%%%%%%%

\bibliographystyle{alpha}
\bibliography{ref}

@article{gross1975logarithmic,
	author = {Gross, Leonard},
	journal = {American Journal of Mathematics},
	number = {4},
	pages = {1061--1083},
	publisher = {JSTOR},
	title = {Logarithmic sobolev inequalities},
	volume = {97},
	year = {1975}}

@book{sarkka2019applied,
	author = {S{\"a}rkk{\"a}, Simo and Solin, Arno},
	publisher = {Cambridge University Press},
	title = {Applied stochastic differential equations},
	volume = {10},
	year = {2019}}

@inproceedings{song2021score,
	author = {Song, Yang and Sohl-Dickstein, Jascha and Kingma, Diederik P and Kumar, Abhishek and Ermon, Stefano and Poole, Ben},
	booktitle = {International Conference on Learning Representations},
	title = {Score-Based Generative Modeling through Stochastic Differential Equations},
	year = {2021}}

@inproceedings{eldan2023noise,
	author = {Eldan, Ronen and Mikulincer, Dan and Raghavendra, Prasad},
	booktitle = {Proceedings of the 55th Annual ACM Symposium on Theory of Computing},
	pages = {661--671},
	title = {Noise stability on the {B}oolean hypercube via a renormalized {B}rownian motion},
	year = {2023}}

@article{gozlan2019deviation,
	author = {Gozlan, N and Madiman, M and Roberto, C and Samson, PM},
	journal = {Journal of Mathematical Sciences},
	number = {4},
	pages = {453--462},
	title = {Deviation Inequalities for Convex Functions Motivated by the {T}alagrand Conjecture},
	volume = {238},
	year = {2019}}

@book{ethier2009markov,
	author = {Ethier, Stewart N and Kurtz, Thomas G},
	publisher = {John Wiley \& Sons},
	title = {Markov processes: characterization and convergence},
	year = {2009}}

@book{pazy2012semigroups,
	author = {Pazy, Amnon},
	publisher = {Springer Science \& Business Media},
	title = {Semigroups of linear operators and applications to partial differential equations},
	volume = {44},
	year = {2012}}

@book{anderson2012continuous,
	author = {Anderson, William J},
	publisher = {Springer Science \& Business Media},
	title = {Continuous-time Markov chains: An applications-oriented approach},
	year = {2012}}

@book{jacod1979calcul,
	author = {Jacod, Jean},
	publisher = {Springer},
	title = {Calcul stochastique et problemes de martingales},
	volume = {714},
	year = {1979}}

@misc{talagrand1989prizes,
	author = {Talagrand, Michel},
	howpublished = {\url{https://michel.talagrand.net/prizes/convolution.pdf}},
	lastchecked = {2016},
	title = {Regularization from ${L}^1$ by convolution},
	year = {2016}}

@inproceedings{follmer2005entropy,
	author = {F{\"o}llmer, Hans},
	booktitle = {Stochastic Differential Systems Filtering and Control: Proceedings of the IFIP-WG 7/1 Working Conference Marseille-Luminy, France, March 12--17, 1984},
	organization = {Springer},
	pages = {156--163},
	title = {An entropy approach to the time reversal of diffusion processes},
	year = {2005}}

@article{ball2013l1,
	author = {Ball, Keith and Barthe, Franck and Bednorz, Witold and Oleszkiewicz, Krzysztof and Wolff, Pawe{\l}},
	journal = {Mathematika},
	number = {1},
	pages = {160--168},
	publisher = {London Mathematical Society},
	title = {L1-smoothing for the {O}rnstein-{U}hlenbeck semigroup},
	volume = {59},
	year = {2013}}

@book{jacod2013limit,
	author = {Jacod, Jean and Shiryaev, Albert},
	publisher = {Springer Science \& Business Media},
	title = {Limit theorems for stochastic processes},
	volume = {288},
	year = {2013}}

@article{gozlan2023log,
	author = {Gozlan, Nathael and Li, Xue-Mei and Madiman, Mokshay and Roberto, Cyril and Samson, P-M},
	journal = {Potential Analysis},
	number = {1},
	pages = {123--158},
	publisher = {Springer},
	title = {Log-Hessian and Deviation Bounds for {M}arkov Semi-Groups, and Regularization Effect in {L}1},
	volume = {58},
	year = {2023}}

@article{eldan2018regularization,
	author = {Eldan, Ronen and Lee, James R},
	journal = {Duke Mathematical Journal},
	number = {5},
	pages = {969--993},
	title = {Regularization under Diffusion and Anti-concentration of the Information Content},
	volume = {167},
	year = {2018}}

@incollection{leonard2012girsanov,
	author = {L{\'e}onard, Christian},
	booktitle = {S{\'e}minaire de Probabilit{\'e}s XLIV},
	pages = {429--465},
	publisher = {Springer},
	title = {Girsanov theory under a finite entropy condition},
	year = {2012}}

@inproceedings{lehec2016regularization,
	author = {Lehec, Joseph},
	booktitle = {Annales de la Facult{\'e} des sciences de Toulouse: Math{\'e}matiques},
	pages = {191--204},
	title = {Regularization in {L}1 for the {O}rnstein-{U}hlenbeck semigroup},
	volume = {25},
	year = {2016}}

@article{talagrand1989conjecture,
	author = {Talagrand, Michel},
	journal = {Israel Journal of Mathematics},
	pages = {82--88},
	publisher = {Springer},
	title = {A conjecture on convolution operators, and a non-{D}unford-{P}ettis operator on ${L}^1$},
	volume = {68},
	year = {1989}}

@book{o2014analysis,
	author = {O'Donnell, Ryan},
	publisher = {Cambridge University Press},
	title = {Analysis of {B}oolean functions},
	year = {2014}}

@inproceedings{eldan2020concentration,
	author = {Eldan, Ronen and Gross, Renan},
	booktitle = {Proceedings of the 52nd Annual ACM SIGACT Symposium on Theory of Computing},
	pages = {208--221},
	title = {Concentration on the {B}oolean hypercube via pathwise stochastic analysis},
	year = {2020}}

\appendix

\section{Additional properties of the Boolean hypercube and the reverse heat process}

\subsection{Additional proof for biased Fourier analysis}
\begin{proof}[Proof of Proposition~\ref{prop:p_biased_analysis}]\label{proof:proof_of_prop1}
  First, we verify that $(\phi_S)_{S \subseteq [n]}$ form an orthonormal basis. We have
  \begin{align*}
    \Exs_{\mu_p}[\phi_i^2] = p_i \frac{(1-p_i)^2}{p_i(1-p_i)} + (1-p_i) \frac{p_i^2}{p_i(1-p_i)} = 1. 
  \end{align*}
  As a result, we also have $\Exs_{\mu_p}[\phi_S^2] = \Exs_{\mu_p} [\prod_{i\in S} \phi_i^2] = \prod_{i\in S} \Exs_{\mu_p} [\phi_i^2] = 1$ because $\mu_p$ is a product measure and $\phi_S$ is in a product form. For the orthogonality, we have
  \begin{align}
    \label{eq:mean_phii_zero}
    \Exs_{\mu_p}\brackets{\phi_i(y) \phi_{\emptyset}} =  \Exs_{\mu_p}\brackets{\phi_i(y)} = p_i \frac{1-p_i}{\sqrt{p_i(1-p_i)}} + (1-p_i) \frac{-p_i}{\sqrt{p_i (1-p_i)}} = 0.
  \end{align}
  Then for $S, R$ two subsets of $[n]$ and $S \neq R$, we have that $S\cap R$ and $(S \cup R) \setminus (S\cap R)$ form a partition of $S \cup R$, and
  \begin{align*}
    \Exs_{\mu_p}[\phi_S \phi_R] &= \Exs_{\mu_p}\brackets{\prod_{i \in S \cap R} \phi_i^2  \prod_{j \in (S \cup R) \setminus (S\cap R)} \phi_j  }\\
    &\overset{(i)}{=} \prod_{i \in S \cap R} \Exs_{\mu_p}\brackets{\phi_i^2} \prod_{j \in (S \cup R) \setminus (S\cap R)} \Exs_{\mu_p} \brackets{ \phi_j} \\
    &\overset{(ii)}{=} 0
  \end{align*}
  (i) follows because $\mu_p$ is a product measure. (ii) follows from Eq.~\eqref{eq:mean_phii_zero} and the fact that since $S \neq R$ the product over $(S \cup R) \setminus (S\cap R)$ is not empty. Hence, $(\phi_S)$ are orthogonal to each other and all have norm $1$. The cardinality of $(\phi_S)$ is $2^n$, which matches the space of real functions on $\braces{-1, 1}^n$. We conclude that $(\phi_S)_{S\subseteq [n]}$ form an orthonormal basis. 

  Second, given a function $f: \braces{-1,1}^n \to \real$, it can be written as 
  \begin{align}
    \label{eq:f_p_biased_decomp}
    f(x) = \sum_{S \subseteq [n]} \hat{f}_p(S) \phi_S(x), \quad \forall x\in \braces{-1,1}^n
  \end{align}
  where $\hat{f}_p(S) = \Exs_{\mu_p}[f \phi_S]$. The above expression coincides with the unique multilinear expansion of $f$, so we can treat $f$ as a polynomial and take derivatives directly. Let $\mathfrak{S} = \braces{i_1, \ldots, i_k} \subseteq [n]$ be a set of indices and $z \in (-1, 1)^n$. Note that when $p = \frac{1+z}{2}$, then $z_i = 2 p_i - 1$ and 
  \begin{align}
    \label{eq:phii_0}
    \phi_i(z) = \frac{1}{2} \parenth{\frac{1-2p_i}{\sqrt{p_i (1-p_i)}} + z_i \frac{1}{\sqrt{p_i (1-p_i)}} } = 0.
  \end{align}
  When taking derivatives of $f$ with respect to $i_1, \ldots, i_k$-th coordinate, any term with $S$ such that $\mathfrak{S} \subseteq S$ in Eq.~\eqref{eq:f_p_biased_decomp} will survive. However, if there are more terms left after taking derivative, say $j \in (S \setminus \mathfrak{S})$, then according to Eq.~\eqref{eq:phii_0} it will be zero. In the end, only $S = \mathfrak{S}$ remains in the sum. Each derivative with respect to $i$ gives a factor of $\frac{1}{2} \frac{1}{\sqrt{p_{i}(1-p_{i})}} = \frac{1}{\sqrt{1-z_i^2}}$. Finally, we obtain 
  \begin{align*}
    \partial_{i_1} \ldots \partial_{i_k} f(z) = \parenth{\prod_{i\in \mathfrak{S}} \frac{1}{\sqrt{1-z_i^2}}} \hat{f}_p(\mathfrak{S}). 
  \end{align*}

\end{proof}

\subsection{Second-order smoothness gained from the heat semigroup}
\begin{lemma}[Hessian bound]
  \label{lem:hessian_smoothness}
  For any nonnegative function $f: \braces{-1,1}^n \to \real_+$ with $f\neq 0$, consider its multilinear expansion. Let $\rho \in (0, 1)$. Then for any $x \in \braces{-\rho, \rho}^n$,
  \begin{align*}
    \nabla^2 \log f(x) \succeq -\frac{1}{(1-\rho)^2} \Ind_n.
  \end{align*}
\end{lemma}
Since $P_\tau f (\cdot) = f(e^{-\tau} \cdot)$, a direct corollary is $\nabla^2 \log P_\tau f(x) \succeq -\frac{e^{-2\tau}}{(1-e^{-\tau})^2} \Ind_n$ for any $x \in \braces{-1,1}^n$. This result can be easily extended to the case where $x \in [-\rho, \rho]^n$. The Hessian bound is reminiscent of the semi-log-convexity gained from the heat semigroup in the Gaussian space (see Section 1.1~\cite{eldan2018regularization}). 
\begin{proof}[Proof of Lemma~\ref{lem:hessian_smoothness}]
  Let $x \in \braces{-\rho, \rho}^n$ and $H \defn \nabla^2 \log f(x)$. Since $f$ is a multilinear polynomial, $\partial_i^2 f = 0$ for every $i$, and we have
  \begin{align*}
    H_{ii} &= - \frac{(\partial_i f)^2}{f^2} \\
    H_{ij} &= \frac{f \partial_{ij} f - \partial_i f \partial_j f}{f^2}, \quad j \neq i.
  \end{align*}
  Take $p = \frac{1+x}{2}$. Using $p$-biased Fourier analysis in Proposition~\ref{prop:p_biased_analysis}, we can write derivatives as Fourier coefficients. 
  \begin{align*}
    f (x) &= \Exs_{\mu_p}[f ] \\
  \partial_i f (x) &= \frac{1}{\sqrt{1-\rho^2}} \hat{f}_p(\braces{i}) = \frac{1}{\sqrt{1-\rho^2}} \Exs_{\mu_p}[f \phi_i] \\
  \partial_{ij} f (x) &=  \frac{1}{1-\rho^2} \hat{f}_p(\braces{i, j}) = \frac{1}{1-\rho^2} \Exs_{\mu_p}[f \phi_i \phi_j],
  \end{align*}
  where $\phi_i$ is defined in Eq.~\eqref{eq:phi_basis_def}. Let $\pi$ be a tilted measure on $\braces{-1,1}^n$, such that for any function $h: \braces{-1,1}^n \to \real$,
  \begin{align*}
    \int h(y) d\pi(y) \defn \frac{\int h(y) f(y) d\mu_p(y)}{\int f(y) d\mu_p(y)}.
  \end{align*}
  Then we have
  \begin{align*}
    H_{ii} &= - \frac{1}{1-\rho^2} \Exs_\pi[\phi_i]^2 \\
    H_{ij} &= \frac{1}{1-\rho^2} \brackets{ \Exs_\pi[\phi_i \phi_j]  - \Exs_\pi[\phi_i] \Exs_\pi[\phi_j]}, \quad j \neq i. 
  \end{align*}
  For any vector $v \in \real^n$, let $R \defn \sum_{i} v_i \phi_i$. From the trivial bound $0 \leq \Var_\pi(R)$, we deduce that
\begin{align*}
  0 &\leq \Exs_{\pi} R^2 - \parenth{\Exs_\pi R}^2 \\
  &= \Exs_{\pi} [\sum_{ij} v_i v_j \phi_i \phi_j ] - \parenth{\sum_{i} v_i \Exs_\pi \phi_i}^2 \\
  &= \sum_{i} v_i^2 (\Exs_\pi[\phi_i^2] - \Exs_\pi[\phi_i]^2) + \sum_{ij, i \neq j } (1-\rho^2) H_{ij} v_i v_j \\
  &= \sum_{i} v_i^2 \Exs_\pi[\phi_i^2] + (1-\rho^2) \sum_i H_{ii} v_i^2 + (1-\rho^2)  \sum_{ij, i \neq j } H_{ij} v_i v_j.
\end{align*}
Hence,
\begin{align*}
  v \tp H v \geq - \frac{1}{1-\rho^2} \sum_{i=1}^n v_i^2 \Exs_\pi[\phi_i^2].
\end{align*}
Finally, for each $i \in [n]$ and $y_i \in \braces{-1,1}$, 
\begin{align*}
  \phi_i(y)^2 \leq \max\braces{\frac{1-p_i}{p_i}, \frac{p_i}{1-p_i}} = \frac{1+\rho}{1-\rho}.
\end{align*}
Therefore,
\begin{align*}
  v \tp H v \geq - \frac{1}{(1-\rho)^2} v\tp v. 
\end{align*}
Since it holds for any $v \in \real^n$, we conclude that $H \succeq -\frac{1}{(1-\rho)^2} \Ind_n$.
\end{proof}

\subsection{Additional properties of the reverse heat process}
\begin{lemma}[Score martingale]
  \label{lem:score_martingale}
  For a strictly positive function $f : \{-1,1\}^n \to (0,\infty)$ with $\vecnorm{f}{1} = 1$ and $(\uV_t)_{t \in [0, T]}$ defined in Eq.~\eqref{eq:def_reverse_process}, for any $i \in [n]$, $(S_i(\rV_t))_{t \in [0, T]}$ is a martingale. Moreover,
  \begin{align*}
    d S_i (\rV_t) = \sum_{j=1}^n \int \Delta_j S_i (\rV_{t-}) \mathbf{1}_{0 < z \leq \frac12 - S_j(\rV_{t-})} \wNj(dt, dz). 
  \end{align*}
  where for $s_j \defn S_j(x), R_{ij} = \frac{x_i x_j \partial_{ij} f (x)}{f(x)}$,
  \begin{align*}
    \Delta_i S_i (x) &= - \frac{2s_i (1-s_i)}{1-2s_i} \\
    \Delta_j S_i (x) &= \frac{2(s_i s_j - R_{ij})}{1-2s_j}, j \neq i.
  \end{align*}
\end{lemma}
This property is reminiscent of the fact that F\"ollmer's drift is a martingale (see e.g., Fact 2.2~\cite{eldan2018regularization}).

\begin{proof}[Proof of Lemma~\ref{lem:score_martingale}]
Recall It\^o's formula for a jump process, for a smooth function $g: [-1,1]^n \to \real$, 
\begin{align*}
  d g(\rV_t) = \angles{\nabla g (\rV_{t-}) , \rV_{t-}} dt + \sum_{j=1}^n \int \Delta_j g(\rV_{t-}) \mathbf{1}_{0 < z \leq \frac12 - S_{j}(\rV_{t-})} \Nj(dt, dz).   
\end{align*}
We apply It\^o's formula to $S_i$. First, note that for $j \in [n]$, 
  \begin{align*}
    \partial_j S_i = \begin{cases}
      \frac{\partial_i f}{f} + x_i \parenth{-\frac{\partial_i f \partial_i f}{f^2}}  & \quad j = i\\
      x_i \parenth{\frac{\partial_{ij} f}{f} - \frac{\partial_i f \partial_j f}{f^2}} & \quad j\neq i.
    \end{cases}
  \end{align*}  
  Hence, the first part in the derivative becomes
  \begin{align*}
    \angles{\nabla S_i(x), x} &= S_i(x) + \sum_{j=1}^n R_{ij}(x) - S_i(x)\sum_j S_j(x), \\
    &= S_i(x) (1-S_i(x)) + \sum_{j=1, j \neq i}^n \parenth{R_{ij}(x) - S_i(x) S_j(x)}
  \end{align*}
  where $R_{ij}(x) = \frac{x_i x_j \partial_{ij} f(x)}{f(x)}$ and $R_{ii} = 0$.
  
  For the jump increments, using that $f$ and $\partial_i f$ are multilinear, we have
  \begin{align*}
    f(x - 2x_j e_j) &= f(x) - 2 x_j \partial_j f(x), \\
    \partial_i f(x - 2x_j e_j) &= \partial_i f(x) - 2 x_j \partial_{ij} f(x), \text{ for } j \neq i \\
    \partial_i f(x - 2x_i e_i) &= \partial_i f(x). 
  \end{align*}
  Hence,
  \begin{align*}
    \Delta_i S_i(x) &= \frac{-x_i \partial_i f(x)}{f(x - 2x_i e_i)} - \frac{x_i \partial_i f(x)}{f(x)} \\
    &= \frac{- S_i(x)}{1-2S_i(x)} - S_i(x) \\
    &= \frac{-S_i(x) (1-S_i(x))}{\frac12 (1-2S_i(x))}. \\
    \Delta_j S_i(x) &= \frac{x_i (\partial_i f(x) - 2 x_j \partial_{ij} f(x) ) }{f(x) - 2 x_j \partial_j f(x)} - S_i(x) \\
    &= \frac{S_i(x) - 2R_{ij}(x)}{1-2S_j(x)} - S_i(x) \\
    &= \frac{S_i(x) S_j(x) - R_{ij}(x)}{\frac12 (1-2S_j(x))}.
  \end{align*}

  Since the $j$-th coordinate jump rate is $\frac12\parenth{1-2S_j(x)}$, the drift coming from centering the jumps cancel exactly with $\angles{\nabla S_i(x), x}$. We obtain that
  \begin{align*}
    dS_i(\rV_t) = \sum_{j=1}^n \int \Delta_j S_i(\rV_{t-}) \mathbf{1}_{0 < z \leq \frac{1}{2}(1-2S_j(\rV_{t-}))} \wNj(dt, dz). 
  \end{align*}
  Hence, $S_i(\rV_t)$ is a local martingale. Since $S_i(\rV_t)$ is also bounded by Lemma~\ref{lem:edge_ratio}, it is a martingale. 
\end{proof}

\end{document}